\documentclass[reqno,11pt,english]{amsart}

\usepackage{latexsym}
\usepackage{amscd}
\usepackage{amssymb}
\usepackage[all, cmtip]{xy}
\usepackage{amsmath}
\usepackage{graphics, graphicx, calc, tikz, tkz-graph, bm, babel, ragged2e, enumerate, etoolbox, xcolor, caption, url}

\usetikzlibrary{decorations.markings, arrows}

\usepackage[pagebackref,colorlinks]{hyperref}
\usepackage{MnSymbol}

\DeclareMathAlphabet{\eusm}{OT1}{eusm}{m}{n}

\newtheorem{theorem}{Theorem}[section]
\newtheorem{proposition}[theorem]{Proposition}
\newtheorem{corollary}[theorem]{Corollary}
\newtheorem{lemma}[theorem]{Lemma}

\theoremstyle{definition}
\newtheorem{definition}[theorem]{Definition}
\newtheorem{example}[theorem]{Example}
\theoremstyle{remark}
\newtheorem{remark}[theorem]{\bf Remark}
\newtheorem{question}[theorem]{\bf Question}

\usetikzlibrary{decorations.markings}
\tikzstyle{vertex}=[circle, draw, fill=black, inner sep=0pt, minimum size=6pt]

}}}]}

\def\Soc{\mbox{Soc\/}}

\def\End{\mbox{End\/}}

\newcommand{\gr}{\operatorname{gr}}
\newcommand{\Ma}{\operatorname{\mathbb M}}

\begin{document}
\title[Graded direct-finiteness and graded $\Sigma$-injective simple modules]{Leavitt path algebras: Graded direct-finiteness and graded $\Sigma$-injective simple modules}
\author{Roozbeh Hazrat}
\address{Centre for Research in Mathematics, Western Sydney University, Australia}
\email{r.hazrat@westernsydney.edu.au}
\author{Kulumani M. Rangaswamy}
\address{Department of Mathematics, University of Colorado at Colorado Springs, Colorado-80918, USA}
\email{krangasw@uccs.edu}
\author{Ashish K. Srivastava}
\address{Department of Mathematics and Statistics, St. Louis University, St.
Louis, MO-63103, USA} \email{ashish.srivastava@slu.edu}

\thanks{The first author would like to acknowledge Australian Research Council grants DP150101598 and DP160101481. A part of this work was done at the University of M\"unster, where he was a Humboldt Fellow. The work of the third author is partially supported by a grant from Simons Foundation (grant number 426367).}
\maketitle

\begin{abstract}
In this paper, we give a complete characterization of Leavitt path algebras
which are graded $\Sigma $-$V$ rings, that is, rings over which a direct sum
of arbitrary copies of any graded simple module is graded injective.
Specifically, we show that a Leavitt path algebra $L$ over an arbitrary
graph $E$ is a graded $\Sigma $-$V$ ring if and only if it is a subdirect
product of matrix rings of arbitrary size but with finitely many non-zero
entries over $K$ or $K[x,x^{-1}]$ with appropriate matrix gradings. We also
obtain a  graphical characterization of such a graded $\Sigma $-$V$ ring $L$%
. When the graph $E$ is finite, we show that $L$ is a graded $\Sigma $-$V$
ring $\Longleftrightarrow L$ is graded directly-finite $\Longleftrightarrow L
$ has bounded index of nilpotence $\Longleftrightarrow $ $L$ is graded
semi-simple. Examples show that the equivalence of these properties in the
preceding statement no longer holds when the graph $E$ is infinite.
Following this, we also characterize Leavitt path algebras $L$ which are
non-graded $\Sigma $-$V$ rings. Graded rings which are graded directly-finite are explored and it is shown that if a Leavitt path algebra $L$ is a graded $\Sigma$-$V$ ring, then $L$ is always graded directly-finite. Examples show the subtle differences between graded and non-graded directly-finite rings. Leavitt path algebras which are graded directly-finite are shown to be directed unions of graded semisimple rings. Using this, we give an alternative proof of a theorem of Va\v{s} \cite{V} on directly-finite Leavitt path algebras.
\end{abstract}

\bigskip

\noindent Leavitt path algebras are algebraic analogues of graph $C^*$-algebras and are also natural generalizations of Leavitt algebras of type $($$1, n$$)$ constructed in \cite{Leavitt}. The first objective of this paper is to characterize graded $\Sigma$-$V$ Leavitt path algebras, namely, those over which direct sum of arbitrary copies of any graded simple module is graded injective. We also describe the larger class of graded directly-finite Leavitt path algebras.  

The initial organized attempt to study the module theory over Leavitt path algebras was done in \cite{AB} where the simply presented modules over a Leavitt path algebra $L_{K}(E)$ of a finite graph $E$ with coefficients in the field $K$, were described.
As an important step in the study of the modules over $L_{K}(E)$ for an
arbitrary graph $E$, the simple $L_{K}(E)$-modules and\ also Leavitt path
algebras with simple modules of special types, have recently been investigated
in a series of papers (see e.g. \cite{AR1}, \cite{C}, \cite{R-1}, \cite{HR}). In this paper we focus on the question when graded simple modules over Leavitt path algebras are not only graded injective but they are graded $\Sigma$-injective and we link this question to the classical ring theoretic properties such as bounded index of nilpotence and direct finiteness.

The study of $V$-rings and $\Sigma$-$V$ rings is a continuation of a long tradition of characterizing rings in terms of certain properties of their cyclic or finitely generated modules (see e.g. \cite{JST}). The origin of this tradition can be drawn back to the celebrated characterization due to Osofsky \cite{O,O1} of semisimple rings as those rings for which cyclic right modules have zero injective dimension, thus obtaining for rings with zero global dimension a dual of Auslander's characterization of global dimension in terms of the projective dimension of cyclic modules. Damiano studied rings over which each proper cyclic right module is injective and showed that such rings are right noetherian \cite{D}. Villamayor was first to study rings over which not necessarily all cyclic but all simple right modules are injective. A ring $R$ is called a right $V$-ring if each simple right $R$-module is injective \cite{Vil}. It is a well-known result due to Kaplansky that a commutative ring is von Neumann regular if and only if it is a $V$-ring. However, in the case of noncommutative setting, the classes of von Neumann regular rings and $V$-rings are quite independent. It is known due to Cartan, Eilenberg, and Bass that a ring $R$ is right noetherian if and only if every direct sum of injective right $R$-modules is injective \cite{Bass}. From this it follows that a ring $R$ is right noetherian if and only if each injective right $R$-module is $\Sigma$-injective. So this observation led to the study of rings over which certain classes of modules are $\Sigma$-injective. Faith intitiated the study of rings over which for each cyclic module $C$, the injective envelope $E(C)$ is $\Sigma$-injective and conjectured that such rings are right noetherian \cite{Faith}. This conjecture is still open. A ring $R$ is called a right $\Sigma$-$V$ ring if each simple right $R$-module is $\Sigma$-injective. These rings were introduced in \cite{GV} and studied subsequently in \cite{Baccellaproc} and \cite{S}. A ring $R$ is called a graded right $\Sigma$-$V$ ring if each graded simple right $R$-module is graded $\Sigma$-injective.  

In this paper, we characterize, both graphically and algebraically, the
Leavitt path algebra $L:=L_{K}(E)$ of an arbitrary graph $E$ which is a
graded $\Sigma $-$V$ ring. Specifically, we show that $L$ is a graded $%
\Sigma $-$V$ ring if and only if it is a subdirect product of matrix rings
of arbitrary size but with finitely many non-zero entries over $K$ or $%
K[x,x^{-1}]$ with appropriate matrix gradings. Examples are constructed
showing that such a subdirect product of matrix rings need not decompose as
a direct sum of matrix rings. We also obtain characterizing graphing
conditions under which $L$ becomes a graded $\Sigma $-$V$ ring. Furthermore,
if $E$ is a finite graph, then $L$ is a graded $\Sigma $-$V$ ring if and
only if $L$ has bounded index of nilpotence if and only if $L$ is graded
directly-finite if and only if $L$ is graded semisimple. Leavitt path
algebras which are non-graded $\Sigma $-$V$ rings are also completely
described. If $L$ is a $\Sigma$-$V$ ring, then we show that $L$ is graded directly-finite. The notions of graded directly-finite rings and graded directly-finite modules are investigated and the subtle distinction from direct-finiteness is illustrated. By considering the matrix ring $\Ma_2 (L(2,3))$ over the Leavitt algebra $L(2,3)$, we show that the two notions of graded directly-finite rings and directly-finite rings are not equivalent, in general. However, they coincide for Leavitt path algebras of arbitrary graphs. Our investigation leads to a structural description of graded directly-finite Leavitt path algebras as being directed unions graded semisimple  subalgebras. Using our approach, we are able to give an alternative and shorter proof of a theorem of Va\v{s} \cite{V} on directly-finite Leavitt path algebras.

\section{Preliminaries}

\noindent For the general notation, terminology and results in Leavitt path algebras, we
refer to \cite{AAS-1}, \cite{R-1} and \cite{T}. We will be using some of the
needed results in associative rings, von Neumann regular rings and modules
from \cite{G} and \cite{L}. We give below an outline of some of the needed basic
concepts and results.

A (directed) graph $E=(E^{0},E^{1},r,s)$ consists of two sets $E^{0}$ and
$E^{1}$ together with maps $r,s:E^{1}\rightarrow E^{0}$. The elements of
$E^{0}$ are called \textit{vertices} and the elements of $E^{1}$
\textit{edges}.

A vertex $v$ is called a \textit{sink} if it emits no edges and a vertex $v$
is called a \textit{regular} \textit{vertex} if it emits a non-empty finite
set of edges. A vertex $v$ is called a \textit{source} if it receives no edges. An \textit{infinite emitter} is a vertex which emits infinitely
many edges. For each $e\in E^{1}$, we call $e^{\ast}$ a \textit{ghost edge}. We let
$r(e^{\ast})$ denote $s(e)$, and we let $s(e^{\ast})$ denote $r(e)$.
A\textit{\ path} $\mu$ of length $n>0$ is a finite sequence of edges
$\mu=e_{1}e_{2}\cdot \cdot \cdot e_{n}$ with $r(e_{i})=s(e_{i+1})$ for all
$i=1,\cdot\cdot\cdot,n-1$. In this case $\mu^{\ast}=e_{n}^{\ast}\cdot
\cdot\cdot e_{2}^{\ast}e_{1}^{\ast}$ is the corresponding \textit{ghost path}. A vertex
is considered a path of length $0$.

A path $\mu$ $=e_{1}\dots e_{n}$ in $E$ is \textit{closed} if $r(e_{n}%
)=s(e_{1})$, in which case $\mu$ is said to be \textit{based at the vertex
}$s(e_{1})$. A closed path $\mu$ as above is called \textit{simple} provided
it does not pass through its base more than once, i.e., $s(e_{i})\neq
s(e_{1})$ for all $i=2,...,n$. The closed path $\mu$ is called a
\textit{cycle} if it does not pass through any of its vertices twice, that is,
if $s(e_{i})\neq s(e_{j})$ for every $i\neq j$.

A graph $E$ is said to satisfy \textit{Condition (K)}, if any vertex $v$ on a
closed path $c$ is also the base of a another closed path $c^{\prime}$ different from $c$.

An \textit{exit }for a path $\mu=e_{1}\dots e_{n}$ is an edge $e$ such that
$s(e)=s(e_{i})$ for some $i$ and $e\neq e_{i}$.

If there is a path from vertex $u$ to a vertex $v$, we write $u\geq v$. A non-empty subset $D$ of vertices is said to be \textit{downward directed }\ if for any
$u,v\in D$, there exists a $w\in D$ such that $u\geq w$ and $v\geq w$. A
subset $H$ of $E^{0}$ is called \textit{hereditary} if, whenever $v\in H$ and
$w\in E^{0}$ satisfy $v\geq w$, then $w\in H$. A hereditary set is
\textit{saturated} if, for any regular vertex $v$, $r(s^{-1}(v))\subseteq H$
implies $v\in H$. For any vertex $v\in E^0$, the set $T_E(v)=\{w\in E^0: v\geq w\}$ is called the \textit{tree} of vertex $v$ and it is the smallest hereditary subset of $E^0$ containing $v$. A \textit{maximal tail }is a subset $M$ of $E^0$ satisfying the following three properties:
\begin{enumerate}

\item $M$ is downward directed, that is, for any pair of vertices $u,v\in M$, there exists a
$w\in M$ such that $u\geq w,v\geq w$;

\item If $u\in M$ and $v\in E^{0}$ satisfies $v\geq u$, then $v\in M$;

\item If $u\in M$ emits edges, there is at least one edge $e$ with $s(e)=u$
and $r(e)\in M$.
\end{enumerate}

A path $p=e_{1}e_{2}\cdot \cdot \cdot e_{n}\cdot \cdot \cdot $ is said
to have a \textit{bifurcation} at a vertex, say $s(e_{k})$, if there is an
edge $f\neq e_{k}$ but $s(f)=s(e_{k})$. In this case $f$ is called a \textit{bifurcating edge}. An infinite path $p$ is called an \textit{infinite
sink} if it contains no bifurcations and no cycles. So $p$ is just a infinite straight line segment $\bullet \longrightarrow \bullet
\longrightarrow \bullet \longrightarrow \bullet \cdots $. A vertex $u$ in $E^{0}$ is called a \textit{line point} if there are neither bifurcations nor cycles at any vertex $w\in T_E(u)$. 

Given an infinite path $p=e_1e_2\cdot \cdot \cdot e_n \cdot \cdot \cdot$ and an integer $n\ge 1$, define $\tau_{\le n}(p)=e_1\cdot \cdot \cdot e_n$ and $\tau_{>n}(p)=e_{n+1}e_{n+2}\cdot \cdot \cdot $. Two infinite paths $p$ and $q$ are said to be \textit{tail-equivalent} if there exist positive integers $m, n$ such that $\tau_{>m}(p)=\tau_{>n}(q)$. An infinite path $p$ is called a \textit{rational path} if $p=ggg \cdot \cdot \cdot $ where $g$ is a (finite) closed path in $E$. An infinite path $q=e_{1}e_{2}\cdot \cdot \cdot e_{n}\cdot \cdot \cdot $ is said to \textit{end in a sink} if there is a $k\geq 1$ such that $e_{k}e_{k+1}\cdot
\cdot \cdot $ is an infinite sink. In other words, $q$ is tail-equivalent to
an infinite sink. An infinite path $p$ is said to \textit{end in a cycle}, if there is a cycle $c$
such that $p$ is tail-equivalent to the rational path $ccc\cdot \cdot \cdot $.

Given an arbitrary graph $E$ and a field $K$, the \textit{Leavitt path algebra
}$L_{K}(E)$ is defined to be the $K$-algebra generated by a set $\{v:v\in
E^{0}\}$ of pair-wise orthogonal idempotents together with a set of variables
$\{e,e^{\ast}:e\in E^{1}\}$ which satisfy the following conditions:

(1) \ $s(e)e=e=er(e)$ for all $e\in E^{1}$;

(2) $r(e)e^{\ast}=e^{\ast}=e^{\ast}s(e)$\ for all $e\in E^{1}$;

(3) (The ``CK-1 relations") For all $e,f\in E^{1}$, $e^{\ast}e=r(e)$ and
$e^{\ast}f=0$ if $e\neq f$;

(4) (The ``CK-2 relations") For every regular vertex $v\in E^{0}$,
\[
v=\sum_{e\in E^{1}, s(e)=v}ee^{\ast}.
\]
Every Leavitt path algebra $L_{K}(E)$ is a $\mathbb{Z}$\textit{-graded algebra}, namely, $L_{K}(E)=\bigoplus\limits_{n\in\mathbb{Z}}L_{n}$ induced by defining, for all $v\in E^{0}$ and $e\in E^{1}$, $\deg
(v)=0$, $\deg(e)=1$, $\deg(e^{\ast})=-1$. Here the $L_{n}$ are abelian
subgroups satisfying $L_{m}L_{n}\subseteq L_{m+n}$ for all $m,n\in\mathbb{Z}$. Further, for each $n\in\mathbb{Z}$, the \textit{homogeneous component }$L_{n}$ is given by
\[
L_{n}=\Big \{%
\sum
k_{i}\alpha_{i}\beta_{i}^{\ast}\in L_K(E):\text{ }|\alpha_{i}|-|\beta_{i}|=n\Big\}.
\]
Elements of $L_{n}$ are called \textit{homogeneous elements}. An ideal $I$ of
$L_{K}(E)$ is said to be a \textit{graded ideal} if $I=$ $%
{\displaystyle\bigoplus\limits_{n\in\mathbb{Z}}}
(I\cap L_{n})$. If $A,B$ are graded modules over a graded ring $R$, we write
$A\cong_{\gr}B$ if $A$ and $B$ are graded isomorphic.
We will also be using the usual grading
of a matrix of finite order. For this and for the various properties of graded
rings and graded modules, we refer to \cite{H-1}, \cite{H} and \cite{NvO}.

A \textit{breaking vertex }of a hereditary saturated subset $H$ is an infinite
emitter $w\in E^{0}\backslash H$ with the property that $0<|s^{-1}(w)\cap
r^{-1}(E^{0}\backslash H)|<\infty$. The set of all breaking vertices of $H$ is
denoted by $B_{H}$. For any $v\in B_{H}$, $v^{H}$ denotes the element
$v-\sum_{s(e)=v,r(e)\notin H}ee^{\ast}$. Given a hereditary saturated subset
$H$ and a subset $S\subseteq B_{H}$, $(H,S)$ is called an \textit{admissible
pair.} Given an admissible pair $(H,S)$, the ideal generated by $H\cup
\{v^{H}:v\in S\}$ is denoted by $I(H,S)$. It was shown in \cite{T} that the
graded ideals of $L_{K}(E)$ are precisely the ideals of the form $I(H,S)$ for
some admissible pair $(H,S)$. Moreover, $L_{K}(E)/I(H,S)\cong L_{K}%
(E\backslash(H,S))$. Here $E\backslash(H,S)$ is a \textit{quotient graph of
}$E$, where $(E\backslash(H,S))^{0}=(E^{0}\backslash H)\cup\{v^{\prime}:v\in
B_{H}\backslash S\}$ and $(E\backslash(H,S))^{1}=\{e\in E^{1}:r(e)\notin
H\}\cup\{e^{\prime}:e\in E^{1}$ with $r(e)\in B_{H}\backslash S\}$ and $r,s$
are extended to $(E\backslash(H,S))^{0}$ by setting $s(e^{\prime})=s(e)$ and
$r(e^{\prime})=r(e)^{\prime}$.

\noindent We will also be using the fact that the Jacobson radical (and in particular,
the prime/Baer radical) of $L_{K}(E)$ is always zero (see \cite{AAS-1}). It is known (see \cite{R-1}) that if $P$ is a prime ideal of $L$ with $P\cap
E^{0}=H$, then $E^{0}\backslash H$ is downward directed.

Let $\Lambda$ be an infinite set and $R$ a ring. Then $\Ma_{\Lambda}(R)$ denotes the ring of $\Lambda \times \Lambda$ matrices in which all except at most finitely many entries are non-zero. 

We will denote by mod-$R$ ($R$-mod), the category of all right (left) $R$-modules. In case of a ring 
$R$ with local units, we consider the full subcategory Mod-$R$ ($R$-Mod) of
the category mod-$R$ ($R$-mod) consisting of right (left) modules $M$ which
are \textit{unital} in the sense that $MR=M$ ($RM=M$). Throughout this paper we will work with unital modules.

Recall that a ring $R$ is said to have \textit{bounded index of} \textit{%
nilpotence} if there is a positive integer $n$ such that $x^{n}=0$ for all
nilpotent elements $x$ in $R$. If $n$ is the least such integer then $R$ is
said to have \textit{index of nilpotence} $n$.

We shall be using some of the results from \cite{AAPS} and \cite{T} where the graphs
that the authors consider are assumed to be countable. We wish to point out
that the results in these two papers hold for arbitrary graphs with no
restriction on the number of vertices or the number of edges emitted by any
vertex. In fact, these results without any restriction on the size of the
graph are stated and proved in \cite{AAS-1}.


\section{Grading of infinite matrices and Leavitt path algebras} \label{grprem} 

\noindent We begin with some basics of graded rings and graded modules that we will need throughout this paper. Let $\Gamma$ be an additive abelian group, $R$ a $\Gamma$-graded ring and $M$ and $N$ graded right
$R$-modules. Consider the hom-group $\operatorname{Hom}_{R}(M,N)$. One can
show that if $M$ is finitely generated or $\Gamma$ is a finite group, then
\[
\operatorname{Hom}_{R}(M,N)=\bigoplus_{\gamma\in\Gamma} \operatorname{Hom}%
(M,N)_{\gamma},
\]
where $\operatorname{Hom}(M,N)_{\gamma}=\{f:M\rightarrow N : f(M_{\alpha
})\subseteq N_{\alpha+\gamma}, \alpha\in\Gamma\}$ (see \cite[\S ~1.2.3]{H}).
However if $M$ is not finitely generated, then throughout we will work with
the group
\[
\operatorname{HOM}_{R}(M,N):=\bigoplus_{\gamma\in\Gamma} \operatorname{Hom}%
(M,N)_{\gamma},
\]
and consequently the $\Gamma$-graded ring
\begin{equation}\label{ferryyy}
\operatorname{END}_{R}(M):=\operatorname{HOM}_{R}(M,M).
\end{equation}

If $f\in\operatorname{END}_{R}(M)$ then $f=\sum_{\gamma\in\Gamma} f_{\gamma}$,
where $f_{\gamma}\in\operatorname{Hom}(M,M)_{\gamma}$ and we write
$\operatorname{supp}(f)=\{\gamma\in\Gamma : f_{\gamma}\not =0\}$.

Let us recall the grading of matrices of finite order and then we will indicate how to extend this to the case of infinite matrices in which at most
finitely many entries are non-zero (see \cite{H-1} , \cite{HR} and \cite{NvO}).

\label{matrixGrading}Let $R$ be a
$\Gamma$-graded ring and $(\delta_{1},\cdot\cdot\cdot,\delta_{n})$ an
$n$-tuple where $\delta_{i}\in\Gamma$. Then $\Ma_{n}(R)$ is a $\Gamma$-graded
ring and, for each $\lambda\in\Gamma$, its $\lambda$-homogeneous component
consists of $n\times n$ matrices%

\begin{equation}\label{one}
\Ma_{n}(R)(\delta_{1},\cdot\cdot\cdot,\delta_{n})_{\lambda}=\left(
\begin{array}
[c]{cccccc}%
R_{\lambda+\delta_{1}-\delta_{1}} & R_{\lambda+\delta_{2}-\delta_{1}} & \cdot
& \cdot & \cdot & R_{\lambda+\delta_{n}-\delta_{1}}\\
R_{\lambda+\delta_{1}-\delta_{2}} & R_{\lambda+\delta_{2}-\delta_{2}} &  &  &
& R_{\lambda+\delta_{n}-\delta_{2}}\\
&  &  &  &  & \\
&  &  &  &  & \\
&  &  &  &  & \\
R_{\lambda+\delta_{1}-\delta_{n}} & R_{\lambda+\delta_{2}-\delta_{n}} &  &  &
& R_{\lambda+\delta_{n}-\delta_{n}}%
\end{array}
\right).
\end{equation}

This shows that for each homogeneous element $x\in R$,
\begin{equation}\label{two}
\deg(e_{ij}(x))=\deg(x)+\delta_{i}-\delta_{j}
\end{equation}
where $e_{ij}(x)$ is a matrix with $x$ in the $ij$-position and with every
other entry $0$.

Now let $I$ be an arbitrary infinite index set and $R$ be a $\Gamma$-graded ring. Denote by $\Ma_{I}(R)$ the matrix with entries indexed by $I\times I$
having all except finitely many entries non-zero and for each $(i,j)\in
I\times I$, the $ij$-position is denoted by $e_{ij}(a)$, where $a\in R$.
Considering a ``vector" $\bar{\delta}:=(\delta_{i})_{i\in I}$, where $\delta
_{i}\in\Gamma$ and following the usual grading on the matrix ring (see equations~(\ref{one}),(\ref{two})), define, for each homogeneous element $a$,%
\begin{equation}\label{three}
\deg(e_{ij}(a))=\deg(a)+\delta_{i}-\delta_{j}.
\end{equation}
This makes $\Ma_{I}(R)$ a $\Gamma$-graded ring, which we denote by
$\Ma_{I}(A)(\bar{\delta})$. Clearly, if $I$ is finite with $|I|=n$, then the usual graded matrix ring $\Ma_n(R)$ coincides (after a suitable permutation) with $\Ma_{n}(R)(\delta
_{1},\cdot\cdot\cdot,\delta_{n})$.

Now, let us translate it in the context of Leavitt path algebras. Suppose $E$ is a finite acyclic graph consisting of exactly one sink $v$. Let
$\{p_{i}:1\leq i\leq n\}$ be the set of all paths ending at $v$. Then it was
shown in (\cite[Lemma 3.4]{AAS-1})
\begin{equation}\label{four}
L_{K}(E)\cong \Ma_{n}(K)
\end{equation}
under the map $p_{i}p_{j}^{\ast}\longmapsto e_{ij}$. Now taking into account
the grading of $\Ma_{n}(K)$, it was further shown in (\cite[Theorem 4.14]{H-1})
that the map (\ref{four}) induces a graded isomorphism
\begin{align}\label{five}
L_{K}(E)&\longrightarrow \Ma_{n}(K)(|p_{1}|,\cdot\cdot\cdot,|p_{n}|),\\
p_{i}p_{j}^{\ast}&\longmapsto e_{ij\text{.}} \notag
\end{align}
In the case of a comet graph $E$ (that is, a finite graph $E$ in which every
path ends at a vertex on a cycle $c$ without exits), it was shown
in \cite{AAS-1} that the map
\begin{align}\label{six}
L_{K}(E)&\longrightarrow \Ma_{n}(K[x,x^{-1}]),\\
p_{i}c^{k}p_{j}^{\ast}&\longmapsto e_{ij}(x^{k})\notag
\end{align}
induces an isomorphism. Again taking into account the grading, it was shown in
(\cite[Theorem 4.20]{H-1}) that the map (\ref{six}) induces a graded isomorphism
\begin{align}\label{seven}
L_{K}(E)&\longrightarrow \Ma_{n}(K[x^{|c|},x^{-|c|}])(|p_{1}|,\cdot\cdot\cdot
\cdot,|p_{n}|),\\
p_{i}c^{k}p_{j}^{\ast}&\longmapsto e_{ij}(x^{k|c|}).\notag
\end{align}
Later in~\cite[Proposition 3.6]{AAPS}, the isomorphisms
(\ref{four}) and (\ref{six}) were extended to infinite acyclic and infinite comet graphs
respectively. The same isomorphisms with
the grading adjustments will induce graded isomorphisms for Leavitt path
algebras of such graphs. We now describe this extension below.

Let $E$ be a graph such that no cycle in $E$ has an exit and such that every
infinite path contains a line point or is tail-equivalent to a rational path
$ccc\cdots$ where $c$ is a cycle (without exits). Define an
equivalence relation in the set of all line points in $E$ by setting $u\sim v$
if $T_{E}(u)\cap T_{E}(v)\neq\emptyset$. Let $X$ be the set of
representatives of distinct equivalence classes of line points in $E$, so that
for any two line points $u,v\in X$ with $u\neq v$, $T_{E}(u)\cap
T_{E}(v)=\emptyset$. For a line point which is not an infinite sink, we choose the sink which is connected to the line point as the representative.  For each vertex \ $v_{i}\in X$, let 
$\overline{{p^{v_{i}}}}:=\{p_{s}^{v_{i}}: 1\leq s \leq n_i\}$ be the set of all paths that end at
$v_{i}$. If $v_i \in X$ is a sink, denote by $|\overline{{p^{v_{i}}}|}=\{|p_{s}^{v_{i}}|: 1\leq s \leq n_i\}$, otherwise 
$|\overline{{p^{v_{i}}}|}=\{\cdots, -2,-1,|p_{s}^{v_{i}}|: 1\leq s \leq n_i\}$.

Let $Y$ be the set of all distinct cycles in $E$. As before, for each cycle
$c_{j}\in Y$ based at a vertex $w_{j}$, let $\overline{{q^{w_{j}}}}
:=\{q_{r}^{w_{j}}:1\leq r\leq m_{j}\}$ be the set of all paths that end at
$w_{j}$ that do not include all the edges of $c_{j}$ where $m_{j}$ is could possibly be infinite. Let $|$ $\overline{{q^{w_{j}}}
}|:=\{|q_{r}^{w_{j}}|: 1\leq r\leq m_{j}\}$. Then the isomorphisms (\ref{five}) and (\ref{seven})
extend to a 
$\mathbb{Z}$-graded isomorphism
\begin{equation}\label{eight}
L_{K}(E)\cong_{\gr}
{\displaystyle\bigoplus\limits_{v_{i}\in X}}
\Ma_{n_{i}}(K)(|\overline{{p^{v_{i}}}}|)\oplus
{\displaystyle\bigoplus\limits_{w_{j}\in Y}}
\Ma_{m_{j}}(K[x^{|c_{j}|},x^{-|c_{j}|}])(|\overline{{q^{w_{j}}}}
|)
\end{equation}
where the grading is as in (\ref{three}).

\section{Graded directly-finite Leavitt path algebras}

\noindent A ring $R$ with identity $1$ is said to be \textit{directly-finite} or \textit{Dedekind finite} if for any
two elements $x,y\in R$, $xy=1$ implies $yx=1$. A ring $R$ with local units is
said to be directly-finite if for every $x,y\in R$ and an idempotent
$u\in R$ such that $ux=xu=x$ and $uy=yu=y$, we have that $xy=u$ implies
$yx=u$. Equivalently, for every local unit $u$, $uR$ is not isomorphic to any
proper direct summand. This is same as saying that the corner ring $uRu$ is a
directly-finite ring with identity.

Note that if $R$ is a unital directly finite ring then it is directly finite in the locally-unital sense as well. Indeed, assuming that $xy = u$ for an idempotent element $u$ with $xu = ux = x$ and $yu = uy = y$ we have that $(x + 1 - u)(y + 1 - u) = xy + 1 - u = 1$. This implies that $1 = (y + 1 - u)(x + 1 - u) = yx + 1 - u$ and from this it follows that $yx = u$~\cite[\S 4]{V}. 

We say a $\Gamma$-graded ring $R$ has \emph{homogeneous local units} if for any finite set of
homogeneous elements  $\{x_{1}, \cdots, x_{n}\}\subseteq R$, there exists a homogeneous
idempotent $e\in R$ such that $\{x_{1}, \cdots, x_{n}\}\subseteq eRe$. Equivalently, $R$
has homogeneous local units if $R_0$ has local units and $R_0
R_{\gamma}=R_{\gamma}R_0=R_{\gamma}$ for every $\gamma \in \Gamma$. 
In general, this is stronger than just having local units, but if $R$ is 
strongly graded then the two notions coincide

A $\Gamma$-graded ring $R$  with unit is called \textit{graded directly-finite}  if for any homogeneous elements $x,y \in R$, $xy=1$ implies $yx=1$. In order to extend the definition to the case of a graded ring with local units, one needs a restriction to a set of homogeneous idempotents (see Example~\ref{memejh}).
A $\Gamma$-graded ring $R$ with homogeneous local units is called \textit{graded directly-finite with respect to a set of homogeneous idempotent S} if the unital graded ring $uRu$, $u \in S$ is graded directly-finite. This means, 
for any two homogeneous elements $x, y\in R^h$ and $u\in S$ such that $xu=x=ux$ and $yu=y=uy$, we have that $xy=u$ implies $yx=u$.  Clearly here $u\in R_0$ and if $x\in R_{\alpha}$ then $y\in R_{-\alpha}$.

Clearly, a directly-finite ring is graded directly-finite but a graded directly-finite ring need not be directly-finite (see Example~\ref{bergmanex}). However in this section we show that in the case of Leavitt path algebras these two concepts coincide (Theorem~\ref{munclassroom}).

Let $R$ be a $\Gamma$-graded ring and $A$, a graded right
$R$-module. Recall that the $\alpha$-\textit{suspension} of the module $A$ is defined as $A(\alpha
)=\bigoplus_{\gamma\in\Gamma}A(\alpha)_{\gamma}$, where $A(\alpha)_{\gamma
}=A_{\alpha+\gamma}$.

\begin{definition}\label{jjuuii}
\label{fgfhr} Let $R$ be a $\Gamma$-graded ring and $A$, a graded right
$R$-module. We say $A$ is \textit{graded directly-finite} if 
$A\cong_{\gr} A(\alpha) \oplus C$, for some $\alpha\in\Gamma$, then
$C=0$.
\end{definition}

The following proposition showcases how the suspensions of a module would come
into play when carrying the results from nongraded case to the graded setting.
Example~\ref{dell} will then show the delicate nature of graded setting and
how things might differ when carrying concepts from nongraded to graded rings. For the following proposition recall the definition of $\operatorname{END}_{R}(A)$ from (\ref{ferryyy}).

\begin{proposition}\label{prop23}
\label{directlyfin} Let $R$ be a $\Gamma$-graded ring and $A$ a graded right
$R$-module. Then $A$ is a graded
directly-finite $R$-module if and only if $\operatorname{END}_{R}(A)$ is a
graded directly-finite ring.
\end{proposition}

\begin{proof}
First suppose $\operatorname{END}_{R}(A)$ is not graded directly-finite. Thus
there are $x \in\operatorname{Hom}_{R}(A,A)_{\alpha}$ and $y \in
\operatorname{Hom}_{R}(A,A)_{-\alpha}$ such that $xy=1_{A}$ but $yx\not =
1_{A}$. Note that $yx$ is a homogeneous idempotent. Now $yA=\bigoplus
_{\gamma\in\Gamma} (yA)_{\gamma}$, where $(yA)_{\gamma}=yA_{\gamma}$, is a
graded submodule of $A(-\alpha)$. Since $xy=1_A$, we have a natural graded
isomorphism $A\rightarrow yA, a\mapsto ya$. On the other hand, we have
\begin{align*}
A & =yA(\alpha)\bigoplus(1_A-yx)A\\
a & = yxa+(1_A-yx)a
\end{align*}
It follows that $A\cong_{\gr}A(\alpha)\bigoplus(1_A-yx)A$ so $A$
is \textit{graded directly-infinite}.

Conversely, if $\theta: A\cong_{\operatorname{gr}} A(\alpha)\oplus B$ for some
$\alpha\in\Gamma$ and $B\not =0$, 
then consider the induced graded projection $x=p\theta: A \rightarrow A(\alpha)\oplus B \rightarrow A(\alpha)$. So $x\in \operatorname{END}(A)_\alpha$. Now define $y=\theta^{-1}i :A(\alpha) \rightarrow A(\alpha)\oplus B \rightarrow A$ which is in  $\operatorname{END}_{R}(A)_{-\alpha}$.  It is easy to see that $xy=1_A$ whereas $yx\not =1_A$ in 
$\operatorname{END}_R(A)$.
\end{proof}

\begin{example}\rm
\label{dell} The following example shows that we do need to take into account
the suspensions of the module $A$ in the Definition~\ref{fgfhr}. We construct a
$\mathbb{Z}$-graded $R$-module $A$ such that $A\cong_{\operatorname{gr}}
A(k)\oplus B$, where $k\not =0$ and $B\not = 0$, $A(n)\ncong
_{\operatorname{gr}} A$ for any $0\not = n \in\mathbb{Z}$. Further we show that if 
$A\cong_{\gr} A \oplus C$ then $C=0$. Indeed we show that
there is a graded decomposition $R\cong_{\operatorname{gr}}R(-1)\oplus R(-1)$,
as graded right $R$-module, however $R \ncong R(n)$ for any $0\not =%
n\in\mathbb{Z}$. For this example, we consider Leavitt's algebra studied
in~\cite{Leavitt}. Consider the free associative $K$-algebra $R$ generated by
symbols $\{x_{i},y_{i} : 1\leq i \leq2\}$ subject to relations
\begin{equation}
\label{jh54320}x_{i}y_{j} =\delta_{ij}, \text{ for all } 1\leq i,j \leq2,
\text{ and } \sum_{i=1}^{2} y_{i}x_{i}=1,
\end{equation}
where $K$ is a field and $\delta_{ij}$ is the Kronecker delta. (This is the Leavitt path algebra associated to the graph with one vertex and two loops). The relations
guarantee the right $R$-module homomorphism
\begin{align}
\label{is329ho}\phi:R & \longrightarrow R^{2}\\
a  & \longmapsto(x_{1}a ,x_{2}a)\nonumber
\end{align}
has an inverse
\begin{align}
\label{is329ho9}\psi:R^{2} & \longrightarrow R\\
(a_{1},a_{2})  & \mapsto y_{1}a_{1}+y_{2}a_{2},\nonumber
\end{align}
so $R\cong R\oplus R$ as a right $R$-module. Thus $R$ is not directly-finite.

Assigning $1$ to $y_{i}$ and $-1$ to $x_{i}$, $1\leq i \leq2$, since the
relations~(\ref{jh54320}) are homogeneous (of degree zero), the algebra $R$ is
a $\mathbb{Z}$-graded algebra. The isomorphism~(\ref{is329ho}) induces a
graded isomorphism
\begin{align}
\label{is329ho22}\phi:R & \longrightarrow R(-1)\oplus R(-1)\\
a  & \mapsto(x_{1}a ,x_{2}a),\nonumber
\end{align}
where $R(-1)$ is the suspension of $R$ by $-1$. However we observe that $R$ is
not graded isomorphic to any nontrivial suspension of itself. For this we need
to use the graded Grothendieck group of $R$ calculated in \cite[\S 3.9.3]{H}. We have
\begin{align*}
K_{0}^{\operatorname{gr}}(R)  & \longrightarrow\mathbb{Z}[1/2]\\
[R(n)]  & \longmapsto2^{n}.
\end{align*}
This shows that $R$ is graded directly-infinite module. Finally if $R\cong_{\gr} R \oplus C$ then $C$ has to be zero as the monoid $V^{\gr}(R)$ is cancellative. 
\end{example}

We will prove that in the case of Leavitt path algebras, graded direct-finiteness coincides with direct-finiteness. However this is not always the case (in fact the counterexample can be chosen from Leavitt's algebras of type different from $(1,n)$). We thank George Bergman for discussion on the construction of a graded directly-finite ring that is not directly-finite which led to the following example. 

\begin{example}\label{bergmanex}
Let $R$ be the Leavitt algebra $L(2,3)$ which is generated by symbols 
$\{x_{i,j},y_{j,i}: 1\leq i \leq 3, 1\leq j \leq 2\}$ subject to the relations $XY=I_3$ and $YX=I_2$, where
\begin{equation} \label{breaktr}
Y=\left( 
\begin{matrix} 
y_{11} & y_{12} & y_{13} \\ 
y_{21} & y_{22}  & y_{23}
\end{matrix} 
\right), \qquad 
X=\left( 
\begin{matrix} 
x_{11} & x_{12} \\ 
x_{21} & x_{22}\\
x_{31} & x_{32}
\end{matrix} 
\right). 
\end{equation} 
For an introduction to such rings see \cite{Bergman} for the universal construction approach and \cite[\S5]{cohn66} for the structure of $L(n,m)$ (which is denoted by $V_{n,m}$). 
The relations (\ref{breaktr}) imply that $R\oplus R\cong R\oplus R\oplus R$. However, we show that $R$ is not isomorphic to a proper direct summand of itself. The monoid of isomorphism classes of finitely generated projective $R$-modules is $\langle u : 2u=3u\rangle$, where $u$ represents the class associated to $R$ (see~\cite[Theorem~5.21]{H-1}). This monoid is isomorphic to the monoid described in the following table. 
\begin{table}[htbp]
\begin{tabular}{l|rclrcl} 
 &      $0$ & $u$ & $v$ \\ \hline 
$0$ & $0$ &$u$ & $v$\\
$u$ &  $u$ &$v$ & $v$\\
$v$ &  $v$ & $v$ & $v$\\
\end{tabular}
\end{table}
One can easily see that if $R$ is isomorphic to a proper direct summand of itself, then $u=u+x$ for some $x\not = 0$ in this monoid, which is not possible. This shows that $R$ is directly-finite. However, as $R\oplus R$ is isomorphic to a proper direct summand of itself, by Proposition~\ref{prop23} (with $\Gamma$ a trivial group), the matrix ring $\Ma_2(R)$ is not directly-finite. Now consider the ring $R$ as a $\mathbb Z$-graded ring concentrated in degree zero (this is not the canonical grading on Leavitt algebras). We have (see Section~\ref{grprem})
\begin{equation*}
\Ma_n(R)(1,2)\cong \End_R\big (R(1)\oplus R(2)\big).
\end{equation*}
We show that the graded $R$-module $R(1)\oplus R(2)$ is graded directly finite (Definition~\ref{jjuuii}). Let
\begin{equation}\label{jhgtgrt}
R(1)\oplus R(2)\cong_{\gr} R(1+n)\oplus R(2+n)\oplus C,
\end{equation}
for some $n\in \mathbb Z$. For any $j\in \mathbb Z$, $(R(1)\oplus R(2))_j=R_{1+j}\oplus R_{2+j}$ is either zero or is one copy of $R$. Since $R$ is not isomorphic to a proper direct summand of itself, 
from Equation~\ref{jhgtgrt} it follows that $C_j=0$, for any $j\in \mathbb Z$, i.e., $C=0$.  Now from Proposition~\ref{prop23}  (with $\Gamma=\mathbb Z$)  it follows that the matrix ring $\Ma_2(R)(1,2)$ is graded directly-finite. So  $\Ma_2(R)(1,2)$ is a graded directly-finite ring that is not directly-finite.
\end{example}

As was shown in the beginning of this section, if a unital ring $R$ is directly-finite, then for any idempotent $u$, $uRu$ is also directly-finite. The following example shows that we can't extend this to the graded directly-finite case and thus justifies the definition of direct-finiteness with respect to a set of homogeneous idempotents. In Theorem~\ref{munclassroom} we show that in fact the case of Leavitt path algebras are quite special and the graded-finiteness (with respect to vertices) implies that the Leavitt path algebra is directly-finite. 

\begin{example}\label{memejh}
Let $F$ be a field and $u,x,y$ are symbols and consider the free unital ring $F\langle u,x,y\rangle$ and the quotient ring 
\[R=F\langle u,x,y\rangle/\langle u^2=u, ux=xu=x,uy=yu=y,xy=u\rangle.\]
Assigning $0$ to $u$, $1$ to $x$ and $-1$ to $y$, since the relations are homogeneous, $R$ becomes a $\mathbb Z$-graded ring. 

We show that this is a unital graded directly-finite ring, however there is a homogeneous idempotent, namely $u$, such that $xy=u$ but $yx\not = u$. 

Using Bergman's diamond machinary~\cite{bergman78} on the generators, one can quickly see that any ambiguity in the relations of the ring can be resolved. Thus $R$ has a basis of the form $y^ix^j$, $u^i$, where $i,j\geq 0$, and $f\in F$. 
This shows that $yx\not =u$. Next we show if $a, b$ are homogeneous and $ab=1$ then $ba=1$. Suppose first $a\in R_l$ where
 $l\not = 0$. Thus $a=\sum f_{i,j} y^ix^j$ and  $b=\sum f_{i',j'} y^{i'}x^{j'}$, where $i\not = j$ and $i'\not = j'$ and $f_i, f_{i'} \in F$.  Therefore $ab$ consists of words of the form $y^kx^k$ and $u^l$. Since these are linearly independent in $R$, their sum can't be $1$. We are left with the case $a\in R_0$. In this case $a=\sum f_{i} y^ix^i+\sum f_{j}u^j+\sum f_{k}$. Notice that all the words in the sum pair-wise commute. Thus if $ab=1$ then $ba=1$ as well. This shows that $R$ is a unital graded directly-finite, but it is not graded directly-finite with respect to the set $S=\{u\}$. Also note that $R$ is not directly finite as 
 $(x+u-1)(y+u-1)=1$ but $(y+u-1)(x+u-1)\not = 1$. 
\end{example}

Next we give an algebraic characterization of graded directly-finite Leavitt path
algebras with respect to the vertices over any arbitrary graph $E$ and, at the same time, give an alternative proof of a theorem of L.
Va\v{s} \cite{V}. The main tool is the subalgebra construction given in \cite{AR}. This
subalgebra construction has turned out to be a very useful tool in making the
\textquotedblleft local-to-global jump\textquotedblright\ while proving a ring
theoretic property of finite character for a Leavitt path algebra $L:=L(E)$ of an
arbitrary graph $E$. This approach has been demonstrated in proving a number
theorems such as the following:

\begin{enumerate}

\item  $L$ is von Neumann regular $\Longleftrightarrow$ $E$ is acyclic
(\cite{AR});

\medskip

\item  Every simple left/right $L$-module is graded
$\Longleftrightarrow$ $L$ is von Neumann regular (\cite{HR});

\medskip

\item  Every Leavitt path algebra $L$ is a graded von Neumann regular ring (\cite{H-2},
\cite{HR});

\medskip

\item Every Leavitt path algebra $L$ is a right/left B\'{e}zout ring
\cite{AMT}. 

\end{enumerate}

Proposition 2 in \cite{AR}, as stated, does not include the additional
properties implied by the subalgebra construction. Indeed, a careful
inspection of the construction in \cite{AR} shows that the morphism $\theta$
in the construction is actually a graded morphism whose image is a graded
submodule of $L$ and it also reveals some properties of cycles.

We include these facts in the following stronger formulation of Proposition 2
of \cite{AR} which we shall be using.

\begin{theorem}
\label{AR} Let $E$ be an arbitrary graph. Then the Leavitt path algebra
$L:=L_{K}(E)$ is a directed union of graded subalgebras $B=A\oplus
K\varepsilon_{1}\oplus$\textperiodcentered\textperiodcentered
\textperiodcentered$\oplus K\varepsilon_{n}$ where $A$ is the image of a
graded homomorphism $\theta$ from a Leavitt path algebra $L_{K}(F_{B})$ to $L$
with $F_{B}$ a finite graph (depending on $B$), the elements $\varepsilon_{i}$
are homogeneous mutually orthogonal idempotents and $\oplus$ denotes a ring
direct sum. Moreover, any cycle $c$ in the graph $F_{B}$ gives rise to a cycle
$c^{\prime}$ in $E$ such that if $c$ has an exit in $F_{B}$ then $c^{\prime}$
has an exit in $E$.
\end{theorem}

We now give an algebraic description of graded directly-finite Leavitt path
algebras. 

\begin{theorem}\label{munclassroom}
\label{DirFinLPAs} Let $E$ be an arbitrary graph. Then the following
properties are equivalent for $L:=L_{K}(E)$:
\begin{enumerate}[\upshape(a)]
\item $L$ is graded directly-finite with respect to the vertices;
\medskip

\item No cycle in $E$ has an exit;
\medskip

\item $L$ is a directed union of graded semisimple Leavitt path algebras;
specifically, $L$ is a directed union of graded direct sums of matrices of
finite order over $K$ or $K[x,x^{-1}]$ with appropriate matrix gradings.

\medskip 

\item $L$ is directly-finite.
\end{enumerate}
\end{theorem}

\begin{proof}
Assume (a). We wish to show that no cycle in $E$ has an exit. Suppose, on the contrary, there is a cycle $c$ based at a vertex
$v$ and has an exit $f$ at $v$. Now both $c$ and $c^{\ast}$ are homogeneous
elements and, by (CK-1) condition, $c^{\ast}c=v$. Then, by hypothesis$_{{}}$,
$cc^{\ast}=v$. Left multiplying the last equation by $f^{\ast}$, we obtain
$0=f^{\ast}cc^{\ast}=f^{\ast}v=f^{\ast}$, a contradiction. Hence no cycle in
$E$ has an exit. This proves (b). 

Assume (b). By Theorem \ref{AR}, $L$ is a directed union of graded subalgebras
$B=A\oplus K\varepsilon_{1}\oplus$\textperiodcentered\textperiodcentered
\textperiodcentered$\oplus K\varepsilon_{n}$, where $A$ is the image of a
graded homomorphism $\theta$ from a Leavitt path algebra $L_{K}(F_{B})$ to $L$
with $F_{B}$ a finite graph depending on $B$. Moreover, any cycle with an exit
in $F_{B}$ gives rise to a cycle with an exit in $E$. Since no cycle in $E$
has an exit, no cycle in the finite graph $F_{B}$ has an exit and so (see Section~\ref{grprem} and (\ref{eight}))
$L_{K}(F_{B})\cong_{\gr}
{\displaystyle\bigoplus\limits_{v_{i}\in X}}
\Ma_{n_{i}}(K)(|\overline{{p^{v_{i}}}}|)\oplus%
{\displaystyle\bigoplus\limits_{w_{j}\in Y}}
\Ma_{m_{j}}(K[x^{|c_{j}|},x^{-|c_{j}|}])(|\overline{{q^{w_{j}}}}|)$, where
$n_{i}$ and $m_{j}$ are positive integers and $X,Y$ are finite sets. Since the
matrix rings $\Ma_{n_{i}}(K)$ and $\Ma_{m_{j}}(K[x,x^{-1}])$ with any matrix grading are graded simple
rings, $A$ and hence $B$ is a direct sum of finitely many matrix rings
of finite order with appropriate matrix gradings over $K$ and/or $K[x,x^{-1}%
]$. This proves (c).

Now (c)$\implies$(d) follows from the known fact that matrix rings $\Ma_{n_{i}}(K)$ and
$\Ma_{m_{j}}(K[x,x^{-1}])$ are directly-finite and finite ring
direct sums of such matrix rings are directly-finite and, by condition
(c), any finite set of elements of $L$ belongs to graded direct sum of finitely
many matrix rings of finite order with appropriate matrix gradings over $K$ or
$K[x,x^{-1}]$. This shows $L$ is directly-finite. 

Obviously (d) implies (a) and this completes the proof.  
\end{proof}

\begin{remark}
The condition that no cycle has an exit appears in study of various other properties of Leavitt path algebras. For example, it is known that the Leavitt path algebra $L_K(E)$ has Gelfand-Kirillov dimension one if and only if $E$ contains at least one cycle, but no cycle in $E$ has an exit.
\end{remark}

\begin{remark} In passing, observe that an arbitrary von Neumann regular ring $R$
need not be directly-finite, as is clear when $R$ is the ring of linear
transformations on an infinite dimensional vector space $V$ over a field $K$.
But if a Leavitt path algebra $L_{K}(E)$ of a graph $E$ is von Neumann regular
then it is always directly-finite due to the fact that $E$ becomes acyclic
(\cite{AR}) and so vacuously satisfies the condition that no cycle in $E$ has
an exit. This also indicates that the ring $R$ of all linear transformations
on an infinite dimensional vector space cannot be realized as the Leavitt path
algebra of a directed graph.
\end{remark}

\noindent Now, we proceed to show that, in a  Leavitt path algebra with identity, the definition of graded direct finiteness with respect to vertices given for Leavitt path algebras without identity is equivalent to the graded direct finiteness defined for Leavitt path algebras with identity. In view of Theorem~\ref{munclassroom}, it is enough if we show that when $L_K(E)$ is a ring with identity and is graded directly finite (that is, for any two homogeneous elements $x,y$, $xy=1$  implies $yx=1$), then no cycle in $E$ has an exit. This is done in the next proposition.

\begin{proposition}
Let $E$ be a graph with finitely many vertices. Then $L(E)$ is graded directly-finite if and only if no cycle in $E$ has an exit. 
\end{proposition}
\begin{proof}
If cycles in $E$ have no exits, by Theorem~\ref{munclassroom}, $L(E)$ is directly-finite and so it is graded directly-finite. 

Let $L(E)$ be a graded directly-finite ring. Suppose first that $E$ has no sources. Consider $E^0=\{v_1,\dots, v_n\}$. For each $i=1, \ldots, n$, select an edge $e_i$ with $r(e_i)=v_i$. Then 
\[(e_1^*+\cdots+e_n^*)(e_1+\cdots+e_n)=1.\] Thus \[(e_1+\cdots +e_n)(e_1^*+\cdots +e_n^*)=1.\] Expanding this equation, a quick inspection gives that each edge has to emit a unique edge and thus the graph has to be a cycle. 

Next suppose that $E$ has a source $v$. Denote by $E_{\backslash v}$, the graph obtained from $E$ by removing $v$ and all the edges with $v$ as a source. We show that if $L(E)$ is graded directly-finite then so is $L(E_{\backslash v})$. (Note that $E_{\backslash v}$ could consist of disjoint graphs.) Suppose $x\in L(E_{\backslash v})_n$ and $y \in L(E_{\backslash v})_{-n}$, where $n$ is an integer such that $xy=1_{L(E_{\backslash v})}$. If $s^{-1}(v)=\{e_1,\dots,e_k\}$, then 
\[\big ( \sum_{l=1}^k e_lxe_l^* +x\big )\big(\sum_{l=1}^k e_l y e_l^* + y\big)=v+1_{L(E_{\backslash v})}=1_{L(E)}.\] 
Since $L(E)$ is graded directly-finite, we have 
\begin{equation}\label{hhggff}
\big(\sum_{l=1}^ke_l y e_l^* + y\big)\big(\sum_{l=1}^k e_lxe_l^* +x\big)=1_{L(E)}.
\end{equation}
Since $1_{L(E_{\backslash v})} 1_{L(E)}=1_{L(E_{\backslash v})}$, a quick inspection of (\ref{hhggff}) shows that $yx=1_{L(E_{\backslash v})}$. 

Now starting from the graph $E$ and repeatedly removing the sources, we end up with graphs which are isolated vertices or have no sources. In the latter case, the first part guarantees the graphs have to be cycles. Thus no cycle in $E$ has an exit. 
\end{proof}

\section{Leavitt path algebras which are graded $\Sigma$-$V$ rings}

\noindent  A ring $R$ over
which every simple left/right module is injective is called a left/right
\textit{$V$-ring}. Here $V$ comes from the name of Orlando Villamayor. If for every simple left/right $R$-module $S$, direct sums of
arbitrary copies of $S$ are injective, then $R$ is called a left/right $\Sigma$-$V$ \textit{ring}. We need to define a graded version of this notion. 

\begin{definition}
Let $R$ be a $\Gamma$-graded ring. A graded left/right $R$-module $M$ is said to be \textit{graded $\Sigma$-injective} if every direct sum $\bigoplus\limits_{i\in I}M(\alpha_{i})$, where $I$ is an arbitrary index set and $\alpha_{i}\in\Gamma$, is graded injective. The
$\Gamma$-graded ring $R$ is called a \textit{graded left/right $\Sigma$-V ring} if every graded simple left/right $R$-module is
graded $\Sigma$-injective.
\end{definition}

In this section, we obtain a complete description of Leavitt path algebras 
of arbitrary graphs which are graded $\Sigma$-$V$ rings. These turn out be
the  subdirect products of matrix rings of arbitrary size, with finitely
many non-zero entries, over $K$ or $K[x,x^{-1}]$ equipped with appropriate
matrix gradings (Theorem~\ref{siginjecmm} and Corollary~\ref{siginjeccc}).

The next proposition shows that the Leavitt path algebras which are $\Sigma$-$V$ rings form a subclass of the class of graded directly-finite Leavitt path algebras.

\begin{proposition}
\label{grSigmaV=>df}Let $E$ be an arbitrary graph. If $L:=L_{K}(E)$ is a
graded $\Sigma$-$V$ ring, then $L$ is graded directly-finite.
\end{proposition}

\begin{proof}
Suppose, on the contrary, $L$ is not graded directly-finite. Then there exists
homogeneous elements $x,y$ and a local unit $v$ such that $xy=v$ but $yx\neq
v$ where $v=v^{2}$ satisfies $vx=x=xv$, $vy=y=yv$. Since $yx$ is a homogeneous
idempotent, we have a decomposition
\begin{equation}\label{mnji}
vL=yxL\oplus(v-yx)L=yL\oplus(v-yx)L, 
\end{equation}
where $yxL=yL$ due to the fact that $yL=yvL=yxyL\subseteq yxL$. Now the map
$va\longmapsto yva=ya$ gives rise to a graded isomorphism $vL\longrightarrow
yL(|y|)$. It now follows from (\ref{mnji}) that $vL\cong_{\gr}A_{1}\oplus_{\gr}B_{1}$, where $A_{1}=yL(|y|)$ and
$B_{1}=(v-yx)L$. As $A_1\cong_{\gr}vL$, we have a graded decomposition
$A_{1}=A_{2}\oplus_{\gr}B_{2}$ such that $A_{2}\cong_{\gr}A_{1}$ and $B_{2}%
\cong_{\gr}B_{1}$. Apply the same arguments to $A_{2}$ to get a graded
decomposition $A_{2}=A_{3}\oplus_{gr}B_{3}$ where $A_{3}\cong_{\gr}A_{2}$ and
$B_{3}\cong_{\gr}B_{2}$. Proceeding like this, we obtain infinitely many
independent graded direct summands of $vL$, $B_{1}\cong_{\gr}B_{2}\cong%
_{\gr}\cdot\cdot\cdot\cong_{\gr}B_{n}\cong_{\gr}\cdots$. For each $i$,
choose a maximal graded $L$-submodule $M_{i}$ in $B_{i}$ in such a way that
$B_{i}/M_{i}\cong_{\gr}B_{j}/M_{j}$ for all $i,j$. Now 
\[(\bigoplus\limits_{n=1}^{\infty}B_n)/(\bigoplus\limits_{n=1}^{\infty}M_n)\cong_{\gr} \bigoplus\limits_{n=1}^{\infty}(B_n/M_n)\]
 is a graded direct sum of isomorphic graded simple $L$-modules,
so it is graded injective, as $L$ is a graded $\Sigma$-$V$ ring. Consequently, $(\bigoplus\limits_{n=1}^{\infty}B_n)/(\bigoplus\limits_{n=1}^{\infty}M_n)$ is a graded direct summand of the graded cyclic module 
$vL/(\bigoplus\limits_{n=1}^{\infty}M_n)$. This means that $\bigoplus \limits_{n=1}^{\infty}(B_{n}/M_{n})$ is a cyclic module, a contradiction. Hence $L$ is graded directly-finite. 
\end{proof}

However, if $L_{K}(E)$ is graded directly-finite, it need not be a graded 
$\Sigma$-$V$ ring as the next example shows.

\begin{example}
\label{gradPrimitive} Consider the following graph $E$, where $E^{0}%
=\{w,v_{1},v_{2},v_{3},\cdots\}$ and
$E^{1}=\{e_{1},e_{2},e_{3},\cdots$; $f_{1},f_{2}%
,f_{3},\cdots$; $g\}$. Further, for all $n\geq1$, $s(e_{n}%
)=v_{n}=s(f_{n})$; $r(e_{n})=v_{n+1}$ ; $r(f_{n})=w$ and $s(g)=w=r(g)$. Thus the graph looks like

 \begin{equation*}
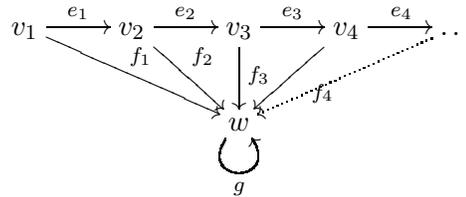
\begin{array}{ccccc}
 \xymatrix{
v_1\ar@{->}[r]^{e_1}\ar@{->}[rrd]^{f_1}& v_2\ar@{->}[r]^{e_2}\ar@{->}[rd]^{f_2}&v_3\ar@{->}[r]^{e_3}\ar@{->}[d]^{f_3}&v_4\ar@{->}[r]^{e_4}   \ar@{->}[ld]^{\,\,\,\,f_4}&\ldots \ar@{.>}[dll]& \\
&&w\ar@(dl,dr)_{g}&&&
}
\end{array}
\end{equation*}

\noindent Clearly $L_{K}(E)$ is graded directly-finite, as the only loop $g$ in $E$ has no
exit. We claim that $L_{K}(E)$ is not a graded $\Sigma$-$V$ ring. To see
this, let $I$ be the graded ideal generated by vertex $w$. By Proposition
3.7 of \cite{AAPS}, $I\cong _{gr}M_{\infty }(K[x,x^{-1}])$ and so $I$ is a
graded direct sum isomorphic graded simple right ideals $S$ of $I$, all
isomorphic to $K[x,x^{-1}]$. Since $I$ is a graded ideal, each such right
ideal $S$ of $I$ is also a graded simple right $L_{K}(E)$-ideal. If $L_{K}(E)
$ were a graded right $\Sigma$-$V$ ring, $I$ will be an injective right ideal
of $L_{K}(E)$ and hence a direct summand of $L_{K}(E)$. But this is a
contradiction, since $v_{i}\geq w$ for all the vertices $v_{i}$ implies, by
Proposition 2.7.10 in \cite{AAS-1}, that $I$ is an essential right ideal of $%
L_{K}(E)$. Thus $L_{K}(E)$ is not a graded $\Sigma$-$V$ ring.
\end{example}

\bigskip 

Before we proceed further, the following observations about prime and
primitive ideals of graded rings may be worth noting:

\noindent Recall that a graded ideal $P$ in a $\Gamma $-graded ring $R$ is graded
prime if for any two homogeneous elements $x,y\in R$, $xRy\subseteq R$
implies $x\in R$ or $y\in R$. A graded ideal $P$ of $R$ is said to be graded
left/right primitive if $P$ is the annihilator of a graded simple left/right 
$R$-module. If $\Gamma $ is an ordered group, then it is known (see
Proposition 1.4, Chapter II, \cite{NvO}) that a graded ideal $P$ of a $%
\Gamma $-graded ring $R$ is graded prime if and only if $P$ is prime. A
similar statement for graded primitive ideals is no longer true. It may
happen that a graded ideal of such a $\Gamma $-graded ring may be graded
primitive, but not primitive. Indeed, consider the graph $E$ in Example \ref%
{gradPrimitive} above. Now $E$ does not satisfy Condition (L), as it
contains the loop $g$ without exits and $E^{0}$ is downward directed. Hence,
by Theorem 4.3, \cite{R-1}, $\{0\}$ is not a primitive ideal and $L_{K}(E)$
is not a primitive ring. We claim that the graded ideal $\{0\}$ is a graded
primitive ideal and so$\ L_{K}(E)$ is a graded primitive ring. To see this,
consider the graded ideal $I$ generated by $w$. As noted in Example \ref%
{gradPrimitive}, $I$ is a graded direct sum isomorphic graded simple right
ideals $S$ of $L_{K}(E)$. Now the annihilator ideal of $S$ in $L_{K}(E)$ is $%
\{0\}$ due to the fact that $I$ is the only non-zero proper (graded) ideal
of $L_{K}(E)$ and $SI=S\neq 0$. Thus $L_{k}(E)$ is a graded primitive ring.

In order to prove our main theorem (Theorem \ref{main Th}), we shall state and prove a series of lemmas and propositions. We first give a graded version of a definition given in \cite{AAPS}. Recall that we will work with unital modules, i.e.,  $MR=M$ for a right (graded) $R$-module $M$.

\begin{definition}
A graded ring $R$ is said to be categorically graded left/right artinian if
every finitely generated graded left/right $R$-module is graded artinian.
\end{definition}

\begin{lemma}\label{munitit}
\label{CatgrArtinian} A Leavitt path algebra $L:=L_K(E)$ of an arbitrary graph $E$ is
categorically graded left (right) artinian if and only if for every vertex $v$,
$Lv$ ($vL$) is graded left (right) artinian.
\end{lemma}

\begin{proof}
Suppose for every vertex $v$, $Lv$ is a graded left 
artinian $L$-module. Since $L=\bigoplus\limits_{v\in E^{0}}Lv$, every finitely generated graded left $L$-module $M$ will be a
graded homomorphic image of $\bigoplus\limits_{i\in I}Lv_{i}(\alpha_i)$, where $I$ is some finite set and $\alpha_i \in \Gamma$. Since $Lv$  is graded left artinian, so is any shifting of $Lv$. This implies $\bigoplus\limits_{i\in I}Lv_{i}(\alpha_i)$ is graded left artinian and so is its homomorphic image $M$. Hence $L$ is categorically
graded left artinian. The proof for the right artinian case  is similar. 
\end{proof}

\noindent A graded ring $R$ is said to be categorically graded left/right noetherian if
every finitely generated graded left/right $R$-module is graded noetherian. A similar proof shows that the Lemma \ref{munitit} holds for categorically left/right noetherian Leavitt path algebras as well.

We next state a graded version of a known property of right ideals in von
Neumann regular rings.

\begin{lemma}
\label{orthoIdemp} Let $R$ be a graded von Neumann regular ring. If $A$ is a
countably generated but not a finitely generated graded right ideal of $R$,
then $A=\bigoplus\limits_{n\geq1}e_{n}A$ where the $e_{n}$ are homogeneous orthogonal idempotents.
\end{lemma}

\begin{proof}
We can easily write $A$ as the union of a strictly ascending chain of finitely
generated graded right ideals $A_{1}\subset A_{2}\subset\cdots\subset A_{n}\subset\cdots$. Since $R$ is graded von
Neumann regular, each $A_{n}$ is a principal ideal generated by a
homogeneous idempotent, say $e$ (\cite{H-2}). We claim that $A_n=eR$ is a graded direct summand of $R$. 
Let $B=\{x-ex:x\in R\}$. We claim that $B$ is a graded right ideal of $R$.
Let $0\neq a\in B$ and suppose $a=a_{n_{1}}+\cdot \cdot \cdot +a_{n_{k}}$ is
a homogeneous decomposition $a$, where each $a_{n_{i}}$ is a homogeneous
element in $R_{n_{i}}$, $n_{i}\in \Gamma $. Then $0=ea=ea_{n_{1}}+\cdot
\cdot \cdot +ea_{n_{k}}$ is a homogeneous sum and so, by the the
independence of the subgroups $R_{n_{i}}$, $ea_{n_{i}}=0$ for all $i=1,\cdot
\cdot \cdot ,k$. Clearly then $a_{n_{i}}=a_{n_{i}}-ea_{n_{i}}\in B$. Thus $B$
is a graded ideal and clearly $eR\oplus B=R$ is a graded decomposition of $R$.

For each
$n$, let $A_{n+1}=A_{n}\oplus C_{n+1}$ be a graded decomposition of $A_{n+1}$.
It is then clear that $A=A_{1}\oplus
{\displaystyle\bigoplus\limits_{n\geq1}}
C_{n+1}$. Now $A_{1}$ and, for $n\geq1$, $C_{n+1}$ are all finitely generated
graded right ideals of $R$ and so $A_{1}=e_{1}R$ and, for all $n\geq1$,
$C_{n+1}=e_{n+1}R$, where $e_{1}, e_2,\cdots$ are all homogeneous idempotents. We then conclude that $A={\displaystyle\bigoplus\limits_{n\geq1}}
e_{n}A$, as desired.
\end{proof}

Next, we proceed to prove the graded version of Lemma 6.17 from \cite{G} for a ring with homogeneous local units. 

\begin{lemma} $($Lemma 2.9 \cite{ARS}$)$ \label{basic1} Suppose $R$ is a $\Gamma $-graded ring
with homogeneous local units. For any graded left $R$-module $M$, the map $\sum_{i=1}^{n}r_{i}\otimes m_{i}\longmapsto \sum_{i=1}^{n}r_{i}m_{i}$, where the $r_{i}$ and the $m_{i}$ are
homogeneous elements, induces a graded isomorphism $R\otimes
_{R}M\longrightarrow M$.
\end{lemma}  

\begin{proof} Since $RM=M$, the given map induces an epimorphism. Suppose $%
\sum_{i=1}^{n}r_{i}m_{i}=0$. Let $\epsilon $ be a homogeneous local
unit in $R$ such that $\epsilon r_{i}=r_{i}=r_{i}\epsilon $ for all $%
i=1,\cdot \cdot \cdot ,n$. Then $\sum_{i=1}^{n}r_{i}\otimes
m_{i}=\sum_{i=1}^{n}\epsilon r_{i}\otimes
m_{i}=\sum_{i=1}^{n}\epsilon \otimes r_{i}m_{i}=\epsilon \otimes
\sum_{i=1}^{n}r_{i}m_{i}=\epsilon \otimes 0=0$. Since tensor product
commutes with direct sums (of homogeneous components), this induces a graded
isomorphism.
\end{proof}

We also need the following Lemma (see also~\cite[49.5]{wisbauer}). 

\begin{lemma} \label{basic2} Suppose $R$ is a $\Gamma $-graded ring with homogeneous local units.
Suppose $M$ is a graded flat left $R$-module. Then, for any graded right
ideal $J$ of $R$, the map $\sum_{i=1}^{n}j_{i}\otimes
m_{i}\longmapsto $ $\sum_{i=1}^{n}j_{i}m_{i}$ where the $j_{i},m_{i}$
are homogeneous elements respectively in $J$ and $M$, induces a graded
isomorphism  $J\otimes _{R}M\longrightarrow JM$.
\end{lemma}

\begin{proof} Note that $JM$ is the subgroup of $M$ generated by $%
\{\sum_{i=1}^{n}j_{i}m_{i}\in M:j_{i}\in J_{\alpha _{i}},m_{i}\in
M_{\beta _{i}},\alpha _{i}, \beta _{i}\in \Gamma \}$. Since $_{R}M$ is flat,
we get an exact sequence 
\[
0\longrightarrow J\otimes _{R}M\longrightarrow R\otimes _{R}M.
\]%
On the other hand by Lemma \ref{basic1} $R\otimes _{R}M \cong M$.  Since the image of the composite map
is $JM$, the stated map is a graded isomorphism $J\otimes
_{R}M\longrightarrow JM$.
\end{proof}

\begin{proposition} \label{6.17} Suppose $\phi :R\longrightarrow S$ is a graded epimorphism between two $\Gamma $%
-graded rings with homogeneous local units and suppose $S$ is graded flat as a graded
left $R$-module. If $A$ is a graded injective right $S$-module, then $A$ is
also graded injective as a right $R$-module.
\end{proposition}

\begin{proof} Let $J$ be a graded right ideal of $R$ and suppose $%
f:J\longrightarrow A$ is a graded morphism of right $R$-modules. Then $f$
induces a graded $S$-module morphism $J\otimes _{R}S\longrightarrow A\otimes_R S$. On the other hand 
$A\otimes_{R}S \cong_{\gr} A\otimes_{S}S \cong_{\gr} A$ by 
Lemma \ref{basic1}. Since $S$ is flat as a left $R$-module, $J\otimes
_{R}S\longrightarrow JS$ is a graded isomorphism by Lemma \ref{basic2}. Hence we get a
graded morphism of right $S$-modules $g:JS\longrightarrow A$ such that $%
g(\phi (x))=f(x)$ for all $x\in J$, noting that $JS=\phi (J)$. Since $A$ is
a graded injective right $S$-module, we get a graded $S$-morphism $%
h:S\longrightarrow A$ such that $h|_{JS}=g$, so that $h(g(\phi (x))=f(x)$ for
all $x\in J$. Then $\theta =h\phi :R\longrightarrow A$ is a graded morphism
such that, for all $x\in J$, $\theta (x)=h(\phi (x))=f(x)$. Hence $A$ is
graded injective as a right $R$-module.
\end{proof}

\begin{definition}
\label{CD} (i) A subset $T$ of vertices in a graph $E$ is said to be \textit{%
countably directed}{\LARGE \ }if there exists a non-empty at most countable
subset $S$ of $T$ such that, for any two $u,v\in T$, there is a $w\in S$
such that $u\geq w$ and $v\geq w$. (Note that, in this case,  $T$ is
trivially downward directed).

(ii) A maximal tail $T$ is called a \textit{special maximal tail} if $T$ is
countably directed.

(iii) \cite{ABR} Given a graph $E$, the set $E^{0}$ is said to have the
countable separation property (for short, CSP), if there is a non-empty at
most countable subset $S\subseteq E^{0}$ such that for each $u\in E^{0}$, there is a $w\in S$ such that $u\geq w$.
\end{definition}

\noindent We are now ready to prove our main theorem. In our proof, two
results play important role: First, Theorem 5.1 in \cite{HR} which states
that if $E$ is an acyclic graph, then every one-sided ideal of $L_{K}(E)$ is
graded and so, in particular, simple $L_{K}(E)$-modules and graded simple $%
L_{K}(E)$-modules coincide. The second result is Theorem 2.4 in \cite{AAPS}
which gives a graphical characterization of the categorically artinian
Leavitt path algebras.

\begin{theorem} \label{siginjecmm} \label{main Th}
Let $E$ be an arbitrary graph. Then the following properties
are equivalent for $L:=L_{K}(E)$:

\begin{enumerate}[(a)]

\item $L$ is a graded right $\Sigma $-$V$ ring;

\item For every graded primitive ideal $P$ of $L$, $L/P$ is categorically
graded artinian;

\item For every graded primitive ideal $P$, $L/P$ is graded isomorphic to $%
M_{\Lambda }(K)(|p^{\bar{v}}|)$ or $M_{\Upsilon }(K[x^{n},x^{-n}])(|q^{\bar{w%
}}|)$ with the matrix gradings indicated in Section \ref{matrixGrading}, where $\Lambda 
$ and $\Upsilon $ are arbitrary index sets and $n$ is a suitable positive
integer.

\item No cycle in $E$ has an exit and, in any special maximal tail $T$ with $%
H=E^{0}\backslash T$, (i) any infinite emitter $u\in T\backslash B_{H}$
satisfies $r(s^{-1}(v))\subseteq H$ and (ii) if an infinite path $p$, with $%
p^{0}\subseteq T$, has infinitely many bifurcations, then almost all the
bifurcating edges $f$ in $p$ satisfy $r(f)\in H$.

\end{enumerate}
\end{theorem}

\begin{proof}
Assume (a). Let $P$ be a graded primitive ideal of $L$, say $P=I(H,S)$.
Suppose, by way of contradiction, $\bar{L}=L/P\cong L_{K}(E\backslash (H,S))$
is not categorically graded right artinian. Then there is a vertex $v\in
E\backslash (H,S)$ such that $v\bar{L}$ is not graded right artinian. This
means that $v\bar{L}$ contains at least one graded right ideal $N$ which is
not finitely generated, because if every graded right ideal inside $v\bar{L}$
is finitely generated, it will be a graded direct summand of $\bar{L}$ and
hence of $v\bar{L}$, as $\bar{L}\cong L_{K}(E\backslash (H,S))$ is graded
von Neumann regular. This means that $v\bar{L}$ is graded semisimple and
since $v\bar{L}$ is cyclic, it will be graded right artinian, a
contradiction. By Lemma \ref{orthoIdemp}, $N$ contains a right ideal $%
A=\dbigoplus\limits_{n\geq 1}e_{n}\bar{L}$ where the $e_{n}$ are homogeneous
idempotents. Let $S$ be a graded simple module annihilated by $P$, so $S$ is
a faithful graded simple $\bar{L}$-module. Consider the module $%
M=\dbigoplus\limits_{n\geq 1}S_{n}$ where $S_{n}\cong _{gr}S$ for all $n$.
Now, for each $n$, the faithful module $S_{n}$ contains an element $x_{n}$
such that $x_{n}e_{n}\neq 0$. For each $n\geq 1$, define a graded
homomorphism $f_{n}:e_{n}\bar{L}\longrightarrow S_{n}$ mapping $%
e_{n}\longmapsto x_{n}e_{n}$ which gives rise to a graded homomorphism $%
f=\oplus f_{n}:\dbigoplus\limits_{n\geq 1}e_{n}\bar{L}\longrightarrow
\dbigoplus\limits_{n\geq 1}S_{n}=M$. By hypothesis, $M$ is injective and so $%
f$ extends to a\ graded homomorphism $g:v\bar{L}\longrightarrow M$. If $%
g(v)=a$, then, on the one hand $g(e_{n})=g(ve_{n})=g(v)e_{n}=ae_{n}$. On the
other hand, $g(e_{n})=f(e_{n})=x_{n}e_{n}\neq 0$. Thus $ae_{n}=x_{n}e_{n}\in
S_{n}$ and non-zero for all $n\geq 1$. This is a contradiction, since $a\in
S_{1}\oplus \cdot \cdot \cdot \cdot \cdot \cdot \oplus S_{k}$ for some $%
k\geq 1$ and this implies that, for any $n>k$, $ae_{n}\in (S_{1}\oplus \cdot
\cdot \cdot \cdot \cdot \cdot \oplus S_{k})\cap S_{n}=0$. Thus $\bar{L}$ is
categorically graded artinian and a primitive ring. This proves (b).

Assume (b). The proof (b) implies (c) is essentially the graded
version of the ideas obtained by combining parts of the proofs of Theorem
2.4 and Theorem 3.7 of \cite{AAPS}. For clarity, we outline the initial
steps of the proof. Let $P=I(H,S)$ be a graded primitive ideal of $L$ where $%
P\cap E^{0}=H$. Since $\bar{L}=L/P\cong L_{K}(E\backslash H,S))$ is
categorically graded artinian, $E\backslash (H,S)$ must be row-finite;
because if $v$ is an infinite emitter in $E\backslash (H,S)$ with $%
s^{-1}(v)\supseteq \{f_{n}:n=1,2,3\cdot \cdot \cdot \}$, then $v\bar{L}$
contains an infinite descending chain of graded right ideals $%
\dbigoplus\limits_{n\geq 1}f_{n}f_{n}^{\ast }\bar{L}\supset
\dbigoplus\limits_{n\geq 2}f_{n}f_{n}^{\ast }\bar{L}\supset
\dbigoplus\limits_{n\geq 3}f_{n}f_{n}^{\ast }\bar{L}\supset \cdot \cdot
\cdot \cdot \cdot $ , a contradiction. Also no path $p$ in $E\backslash
(H,S) $ contains infinitely many bifurcations; because, otherwise, we write $%
p$ as a concatenation of paths $p=p_{1}p_{2}p_{3}\cdot \cdot \cdot \cdot $
where, for each $n$, there is a bifurcating edge $e_{n}$ with $%
s(e_{n})=r(p_{n})$. If $v=s(p)$, then one can verify that $v\bar{L}$
contains an infinite descending chain of graded right ideals 
\begin{equation*}
p_{1}p_{1}^{\ast }\bar{L}\supsetneqq p_{1}p_{2}p_{2}^{\ast }p_{1}^{\ast }%
\bar{L}\supsetneqq p_{1}p_{2}p_{3}p_{3}^{\ast }p_{2}^{\ast }p_{1}^{\ast }%
\bar{L}\supsetneqq \cdot \cdot \cdot \cdot \cdot
\end{equation*}%
contradicting that $v\bar{L}$ is graded right artinian. Likewise, if there
is a cycle $c$ with an exit $f$ at a vertex $v$ in $\bar{L}$, then again $v%
\bar{L}$ will contain the infinite descending chain of graded right ideals 
\begin{equation*}
cc^{\ast }\bar{L}\supsetneqq c^{2}(c^{\ast })^{2}\bar{L}\supsetneqq c^{3}(c^{\ast
})^{3}\bar{L}\cdot \cdot \cdot \cdot
\end{equation*}%
which is a contradiction. Thus no cycle in $E\backslash (H,S)$ has an exit.
This means that every path in $E\backslash (H,S)$ ends at a line point or at
a cycle without exits. Also $E\backslash (H,S)$ is row-finite. Moreover,
since $E^{0}\backslash H$ is downward directed, either $E\backslash (H,S)$
contains a unique finite or infinite sink and no cycles or $E\backslash
(H,S) $ contains a unique cycle $c$ without exits. Downward directness of $%
E^{0}\backslash H$ also implies, either all the paths in $E\backslash (H,S)$
end at the unique finite or infinite sink or all the paths in $E\backslash
(H,S)$ end at the unique cycle $c$ (without exits). Then by Theorem 3.7 of 
\cite{AAPS},%
\begin{equation*}
L_{K}(H,S))\cong M_{\Lambda }(K)\text{ or }L_{K}(E\backslash (H,S))\cong
M_{\Upsilon }(K[x^{|c|},x^{-|c|}])
\end{equation*}%
which is a graded isomorphism with the matrix gradings as given in Section
2, where $\Lambda $ ($\Upsilon $) denotes the set of all paths in $%
E\backslash (H,S)$ that end at the unique sink (the set of all paths that
end at $c$ but do not contain the entire cycle $c$). This proves (c).

Assume (c). Let $S$ be a graded simple right $L$ module and let $P$ be the
right annihilator of $S$ in $L$. Then $P$ is a graded primitive ideal and
so, by hypothesis, either $L/P\cong _{gr}M_{\Lambda }(K)(|\overset{-}{p^{%
\bar{v}}|)\text{ or }}M_{\Upsilon }(K[x^{|n|},x^{-|n|}])(|q^{\bar{w}}|)$.
Consequently, $L/P$ is graded semi-simple \cite{HR}. In particular, as a
right $L/P$-module, $S$ is graded $\Sigma $-injective. Since $L$ is graded
von Neumann regular, $L/P$ is graded flat as a graded right $L$-module. Then
by Proposition 4.9, $S$ is $\Sigma $-injective as a graded right $L$-module.
This proves (a).

Assume (a) (and (b),(c)). By Proposition 4.2, no cycle in the graph $E$ has
an exit. Let $T$ be a special maximal tail in $E$ and let $H=E^{0}\backslash
T$. Consider $E\backslash (H,B_{H})$.

Case 1: Suppose $E\backslash (H,B_{H})$ contains a cycle $c$ based at a
vertex $w$. Since $c$ has no exit and since $(E\backslash (H,B_{H}))^{0}=T$
is, in particular,  downward directed, $c$ is the only cycle in $E\backslash
(H,B_{H})$ and $u\geq w$ for all vertices $u\in E\backslash (H,B_{H})$. Let $%
I$ be the two-sided graded ideal of $\bar{L}=L_{K}(E\backslash (H,B_{H}))$
generated by $c^{0}$. By the graded version of Lemma 2.7.1 of \cite{AAS-1} and
using the matrix gradings indicated in Section 2 , $I\cong _{gr}M_{\Upsilon
}(K[x^{n},x^{-n})(|q^{\bar{w}}|)$, where $\Upsilon $ is the set of all paths
that end at $w$ but do not include the entire cycle $c$ and $n$ is the
length of the cycle $c$. Thus, as an $I$-module, $I$ is a direct sum of
isomorphic graded simple right ideals of $I$, each of which is graded
isomorphic to $K[x,x^{-1}]$. Since the graded ideal $I$ is a ring with local
units \cite{AAPS}, these isomorphic graded simple right ideals of $I$ are
also isomorphic graded simple right ideals of $\bar{L}$. As $\bar{L}$ is a
graded $\sum $-$V$-ring, $I$ is then a graded injective right $\bar{L}$%
-module and hence is a direct summand of $\bar{L}$. On the other hand, since 
$u\geq w$ for every $u\in E\backslash (H,B_{H})^{0}$, Proposition 2.7.10 in 
\cite{AAS-1} implies that $I$ is an essential right ideal of $\bar{L}$.
Consequently, $\bar{L}=I\cong _{gr}M_{\Upsilon }(K[x^{n},x^{-n})(|q^{\bar{w}%
}|)$. By Theorem 3.7 of \cite{AAPS}, $E\backslash (H,B_{H})$ is then row
finite and every infinite path is tail equivalent to the rational path $%
ccc\cdot \cdot \cdot $. Now if $T$ contains an infinite emitter $u\notin
B_{H}$, $u$ must be an infinite emitter or a sink in $E\backslash (H,B_{H})$%
. As $E\backslash (H,B_{H})$ is row-finite, $u$ is then a sink in $%
E\backslash (H,B_{H})$ and this means that, $r_{E}(s_{E}^{-1}(u))\subseteq H
$. Suppose $p$ is an infinite path  with $p^{0}\subseteq T$ and has
infinitely many bifurcating edges in $E$. Since $p$ is an infinite path in $%
E\backslash (H,B_{H})$, as noted above, $p$ is then  tail-equivalent to the
rational path $ccc\cdot \cdot \cdot $. This means that  $r(f)\in H$ for
almost all the bifurcating edges $f$ of $p$. Thus $T$ has the desired
properties. 

Case 2: Suppose $E\backslash (H,B_{H})$ contains no cycles. Let $P=I(H,B_{H})
$. Now $E\backslash (H,B_{H})$, being acyclic, trivially satisfies Condition
(L).  Also $(E\backslash (H,B_{H}))^{0}=T$ is countably directed and hence $%
(E\backslash (H,B_{H}))^{0}$ is not only downward directed but also
satisfies the CSP. So $\bar{L}=L_{K}(E\backslash (H,B_{H}))$ is a primitive
ring, by Proposition 4.8 in \cite{ABR}. On the other hand, since $%
E\backslash (H,B_{H})$ is acyclic, Theorem 5.1 in \cite{HR} implies that the
set of simple right $\bar{L}$-modules coincides with the set of graded
simple right $\bar{L}$-modules. Consequently, $L/P\cong _{gr}\bar{L}$ is a
graded primitive ring and thus $P$ is a graded primitive ideal of $L$. By
(b), $L/P$ is categorically graded artinian. Since $E\backslash (H,B_{H})$
is acyclic, every one-sided ideal of $L_{K}(E\backslash (H,B_{H}))$ is
graded by Theorem 5.1 in \cite{HR}. Consequently, $L_{K}(E\backslash
H,B_{H}))$ is categorically artinian. We then appeal to Theorem 2.4 in \cite%
{AAPS} to conclude that $E\backslash (H,B_{H})$ is row finite and every
infinite path is tail equivalent to an infinite sink. This means, as in Case
1, if $T$ contains an infinite emitter $u\notin B_{H}$, then $%
r_{E}(s_{E}^{-1}(u))\subseteq H$. Also if an infinite path $p$ with $%
p^{0}\subseteq T$ has infinitely many bifurcating edges in $E$, then, since 
$p$ is an infinite path in $E\backslash (H,B_{H})$ tail-equivalent to an
infinite sink, $r(f)\in H$ for almost all the bifurcating edges $f$ of $p$.
This proves (d).

Assume (d). Let $P=I(H,S)$ be a graded primitive ideal of $L$. By Theorem
4.3 in \cite{R-1}, $T=E^{0}\backslash H$ is downward directed and satisfies
the CSP and so $T$ is countably directed and is a special maximal tail.

Case 1: Suppose $E\backslash (H,S)$ contains no cycles. Now, by hypothesis
(d), $r(s^{-1}(v))\subseteq H$ for every infinite emitter $u\in T\backslash
B_{H}$ and if an infinite path $p$ with $p^{0}\subseteq T$ has infinitely
many bifurcations in $E$, then almost all the bifurcating edges $f$ \ in $p$
satisfy that $r(f)\in H$. This means that $E\backslash (H,S)$ is a
row-finite graph and every infinite path is tail-equivalent to an infinite
sink which is unique due to downward directedness of $E\backslash (H,S)$.
Since $E\backslash (H,S)$ is further acyclic,  we appeal to Theorem 2.4 in 
\cite{AAPS}, to conclude that $L_{K}(E\backslash (H,S))$ is categorically
artinian. Now the fact that $E\backslash (H,S)$ contains no cycles implies,
by Theorem 5.1 in \cite{HR}, that every one-sided ideal of $%
L_{K}(E\backslash (H,S))$ is graded and so $L_{K}(E\backslash (H,S))$ is
categorically graded artinian.

Case 2: Suppose $E\backslash (H,S)$ and hence $T$ contains a cycle $c$ based
at a vertex $w$. Since $T$ is downward directed and since $c$ has no exits
in $E\backslash (H,S)$, $u\geq w$ for every $u\in T$. Thus $T$ is countably
directed with respect to $\{w\}$. Proceeding as before, the hypothesis on $%
T$ implies that $E\backslash (H,S)$ is a row-finite graph with a unique
cycle $c$ and every infinite path in $E\backslash (H,S)$ is tail-equivalent
to the rational path $ccc\cdot \cdot \cdot $. Then, by Theorem 3.7 in \cite%
{AAPS}, $L_{K}(E\backslash (H,S))\cong M_{\Upsilon }(K[x^{n},x^{-n})$ which
is a graded isomorphism under the matrix grading indicated in Section 2.
Thus we conclude that $L/P\cong _{gr}L_{K}(E\backslash (H,S))\cong
_{gr}M_{\Upsilon }(K[x^{n},x^{-n})(|q^{\bar{w}}|)$ which is categorically
graded artinian, as $M_{\Upsilon }(K[x^{n},x^{-n})(|q^{\bar{w}}|)$ is graded
semi-simple. This proves (b).
\end{proof}

\bigskip

When the graph $E$ is row-finite, the tree $T_{E}(v)$ is at most countable
for any vertex $v$ in $E$. So if $T$ is a non-empty maximal tail and $v\in T$%
, then $T$ will be countably directed with respect to the at most countable
set $T\cap T_{E}(v)$ and hence is a special maximal tail. In this case, the
condition (d) in Theorem \ref{main Th} gets simplified as noted below.

\begin{corollary}
\label{Row-finite} Let $E$ be a row-finite graph. Then $L_{K}(E)$ is
graded $\Sigma $-$V$ ring if and only if no cycle in $E$ has an exit and if $%
T$ is a maximal tail in $E$ with $H=E^{0}\backslash T$, then any infinite
path $p$ with $p^{0}\subseteq T$ and having infinitely many bifurcations in 
$E$ has almost all its bifurcating edges $f$ satisfying $r(f)\in H$.
\end{corollary}

Now the intersection of all graded primitive ideals of $L_{K}(E)$, being the
graded Jacobson radical of $L_{K}(E)$, is zero (see \cite{H-2}) and so we get
the following as a consequence of Theorem \ref{siginjecmm}.

\begin{corollary}\label{siginjeccc}
If for arbitrary graph $E$, $L:=L_{K}(E)$ is a graded $\Sigma$-$V$ ring, then
$L$ is a graded subdirect product of matrix rings of arbitrary size but with
finitely many non-zero entries over $K$ or $K[x,x^{-1}]$ equipped with
appropriate matrix gradings.
\end{corollary}

If $L_K(E)$ is a graded $\Sigma$-$V$ ring, $L_K(E)$ need not decompose as a graded
direct sum of matrix rings, as the following example shows.

\begin{example} 
\label{InfiniteClock} Consider the following ``infinite clock" graph $E$:
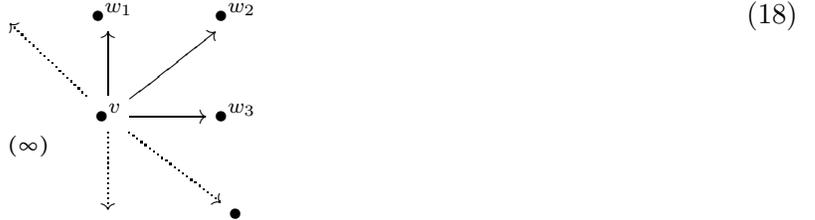
\begin{equation}\label{inififi}
\xymatrix{ & \, \, {\bullet}^{w_1} & {\bullet}^{w_2} \\
 & {\bullet}^v \ar@{.>}[ul] \ar[u] \ar[ur] \ar[r] \ar@{.>}[dr] \ar@{.>}[d]
\ar@{}[dl] _{(\infty)} & {\bullet}^{w_3} \\ &  & {\bullet}} 
\end{equation}

Thus $E^{0}=\{v\}\cup\{w_{1},w_{2},\cdots\}$, 
where the $w_{i}$ are all sinks. For each $n\geq1$, let $e_{n}$ denote the
single edge connecting $v$ to $w_{n}$. The graph $E$ is acyclic and so every
ideal of $L$ is graded. The number of distinct paths ending at any given sink
\ (including the sink) is $\leq2$. For each $n\geq1$, $H_{n}=\{w_{i}:i\neq
n\}$ is a hereditary saturated set, $B_{H_{n}}=\{v\}$ and $E^{0}\backslash
H_{n}=\{v,w_{n}\}$ is downward directed. Hence the ideal $P_{n}$ generated by
$H_{n}\cup\{v-e_{n}e_{n}^{\ast}\}$ is a\ graded primitive ideal and
$L_{K}(E)/P_{n}\cong \Ma_{2}(K)$. Moreover, every graded (prime) primitive ideal
$P$ of $L_{K}(E)$ is equal to $P_{n}$ for some $n$. By Theorem \ref{main Th},
$L_{K}(E)$ is a\ graded $\Sigma$-$V$ ring.

But $L_{K}(E)$ cannot be decomposed as a direct sum of the matrix rings $\Ma_{2}%
(K)$. Because, otherwise, $v$ would lie in a direct sum of finitely many
copies of $\Ma_{2}(K)$. Since the ideal generated by $v$ is $L_{K}(E)$,
$L_{K}(E)$ will then be a direct sum of finitely many copies of $\Ma_{2}(K)$.
This is impossible since $L_{K}(E)$ contains an infinite set of orthogonal
idempotents $\{e_{n}e_{n}^{\ast}:n\geq1\}$.

We can also describe the internal structure of this ring $L_{K}(E)$. The socle
$S$ of $L_{K}(E)$ is the ideal generated by the sinks $\{w_{i}:i\geq1\}$,
$S\cong%
{\displaystyle\bigoplus\limits_{\aleph_{0}}}
\Ma_{2}(K)$ and $L_{K}(E)/S\cong K$.
\end{example}

If $E$ is a row-finite or a countable graph, then again the graded $\Sigma $-$V
$ ring $L_{K}(E)$ in Theorem \ref{Row-finite} need not decompose as a 
direct sum of matrix rings $\Ma_{\Lambda }(K)(|\overline{q^w}|)$ and $\Ma_{\Upsilon
}(K[x^{n},x^{-n}])(|\overline{q^w}|)$ as the next example shows.

\begin{example}
\label{row-finite/countable graph} Consider the following ``$\mathbb N \times \mathbb N$-Lattice" row-finite graph $F$, where the
vertices in $F$ are points in the first quadrant of the coordinate plane
whose coordinates are integers $\geq 0$. Specifically, $F^{0}=\{(m,n):m,n%
\geq 0\}$.

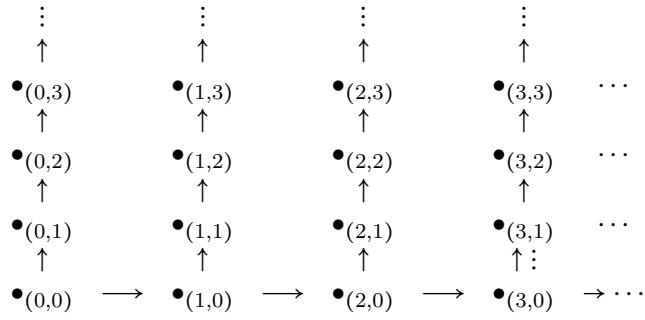
\begin{equation*}
\begin{array}{cccccccc}
\vdots  &  & \vdots  &  & \vdots  &  & \vdots  &  \\ 
\uparrow  &  & \uparrow  &  & \uparrow  &  & \uparrow  &  \\ 
\bullet _{(0,3)} &  & \bullet _{(1,3)} &  & \bullet _{(2,3)} &  & \bullet
_{(3,3)} & \cdot \cdot \cdot  \\ 
\uparrow  &  & \uparrow  &  & \uparrow  &  & \uparrow  &  \\ 
\bullet _{(0,2)} &  & \bullet _{(1,2)} &  & \bullet _{(2,2)} &  & \bullet
_{(3,2)} & \cdot \cdot \cdot  \\ 
\uparrow  &  & \uparrow  &  & \uparrow  &  & \uparrow  &  \\ 
\bullet _{(0,1)} &  & \bullet _{(1,1)} &  & \bullet _{(2,1)} &  & \bullet
_{(3,1)} & \cdot \cdot \cdot  \\ 
\uparrow  &  & \uparrow  &  & \uparrow  &  & \uparrow \vdots  &  \\ 
\bullet _{(0,0)} & \longrightarrow  & \bullet _{(1,0)} & \longrightarrow  & 
\bullet _{(2,0)} & \longrightarrow  & \bullet _{(3,0)} & \rightarrow \cdot
\cdot \cdot 
\end{array}%
\end{equation*}%
Now $F$ is acyclic and so, by Theorem 5.1 \cite{HR}, every one sided ideal
of $L_{K}(F)$ is graded. It is easy to verify that there are only two
(graded) prime (primitive) ideals in $L_{K}(F)$, namely, the ideal $A$
generated by all points not on the X-axis, namely, $A=\langle (m,n):n\neq 0 \rangle$
and the ideal $B$ generated by all the points not on the Y-axis, namely, $%
B=\langle (m,n):m\neq 0  \rangle $ (Because, both $F^{0}\backslash A$ and $%
F^{0}\backslash B$ are downward directed). Now both $L_{K}(F)/A$ and $%
L_{K}(F)/B$ are graded isomorphic to $\Ma_{\infty }(K)$ with appropriate
matrix gradings. Hence, by Corollary \ref{Row-finite}, $L_{K}(F)$ is a graded $\Sigma$-$V$ ring. Note that the line points in the graph $F$ are all the
points not on the $X$-axis and they form a hereditary saturated set. It is
known (\cite{AAS-1}) that the ideal generated by line points is the socle 
$\Soc(L_{K}(F))$. Clearly, $(0,0)\notin $ $\Soc(L_{K}(F))$. We claim that $%
L_{K}(F)$ cannot decompose as a direct sum of matrix rings  graded
isomorphic to $\Ma_{\infty }(K)$. Because, otherwise,  $\Soc(\Ma_{\infty
}(K))=\Ma_{\infty }(K)$ will imply $\Soc(L_{K}(F))=L_{K}(F)$, a contradiction
since $(0,0)\notin $ $\Soc(L_{K}(F))$.
\end{example}

\noindent However, as we shall see next, if the graph $E$ is finite and $L:=L_{K}(E)$
is a graded $\Sigma$-$V$ ring, then $L$ admits the desired decomposition.

\begin{theorem}\label{siginjec}
\label{Finite graph}Let $E$ be a finite graph. Then for  $L:=L_{K}(E)$ the following are
equivalent. 
\begin{enumerate}[\upshape(a)]
\item $L$ is a graded $\Sigma$-$V$ ring;
\medskip

\item No cycle in $E$ has an exit;
\medskip 

\item $L$ is (graded) directly-finite;

\medskip 

\item $L\cong_{\gr}
{\displaystyle\bigoplus\limits_{v_{i}\in X}}
\Ma_{n_{i}}(K)(|\overline{p^{v_{i}}}|)\oplus%
{\displaystyle\bigoplus\limits_{w_{j}\in Y}}
\Ma_{m_{j}}(K[x^{|c_{j}|},x^{-|c_{j}|}])(|\overline{q^{w_{j}}}|)$, where
$n_{i}$ and $m_{j}$ are positive integers, $X, Y$ are finite sets being the set
of sinks and the set of all cycles in $E$ respectively;

\medskip

\item $L$ is graded left/right semi-simple, that is, $L$ is a graded direct sum
of (finitely many) graded simple left/right $L$-modules;

\medskip

\item $L$ has bounded index of nilpotence;

\medskip

\item $L$ is a graded self-injective ring.
\end{enumerate}
\end{theorem}

\begin{proof}
Now, by Proposition \ref{grSigmaV=>df} and Theorem \ref{munclassroom}, 
(a) $\Longrightarrow$
(b) and (b)
$\Longleftrightarrow$
(c). The equivalence (d)
$\Longleftrightarrow$
(g) was shown in \cite{HR}.

Assume (b). Since $E$ is a finite graph, every path in $E$ eventually ends
at a sink or is tail equivalent to a rational path $ccc\cdots$ for
some cycle $c$ without exits. Then, by \cite[Theorem 3.7]{AAS-1},
\begin{equation}\label{star}
L\cong
{\displaystyle\bigoplus\limits_{v_{i}\in X}}
\Ma_{n_{i}}(K)\oplus
{\displaystyle\bigoplus\limits_{c_{j}\in Y}}
\Ma_{m_{j}}(K[x,x^{-1}]),
\end{equation}
where $X$ and $Y$ are finite sets denoting respectively the set of sinks
and the set of distinct cycles in $E$ and for each $i$, $n_{i}$, denotes
the number of paths ending at the sink $v_{i}$ and for each $j$,  $m_j$ denotes the number
of paths tail equivalent to the rational path $c_{j}c_{j}c_{j}\cdots
$. As established in Section~\ref{grprem}, the matrices in (\ref{star}) can be given
appropriate matrix gradings giving rise to a graded isomorphism%
\begin{equation}\label{star2}
L\cong_{\gr}
{\displaystyle\bigoplus\limits_{v_{i}\in X}}
\Ma_{n_{i}}(K)(|\overline{p^{v_{i}}}|)\oplus%
{\displaystyle\bigoplus\limits_{w_{j}\in Y}}
\Ma_{m_{j}}(K[x^{|c_{j}|},x^{-|c_{j}|}])(|\overline{q^{w_{j}}}|).
\end{equation}
This proves (d).

Assume (d). The matrices $\Ma_{n_{i}}(K)(|\overline{p^{v_{i}}}|)$ and
$\Ma_{n_{j}}(K[x^{|c_{j}|},x^{-|c_{j}|}])(|\overline{q^{w_{j}}}|)$ are both
direct sums of graded simple left/right modules (see \cite{HR}).
Consequently, $L$ is graded left/right semi-simple, thus proving (e).

Assume (e). Since $L$ is graded left/right semi-simple, every graded
left/right $L$-module is injective. In particular, $L$ is a graded $\sum$%
-$V$ ring, thus proving (a).

Assume (d). Now a matrix ring of finite order $n$ over $K$ or $K[x,x^{-1}]$
has bounded index of nilpotence $n$ (see \cite{G}). Since the index sets $\Lambda$
and $\Upsilon$ in (\ref{star2}) are finite, we then conclude that $L$ has bounded
index of nilpotence, thus proving (f).

Assume (f), so $L$ has bounded index of nilpotence , say $n$. Suppose, by
way of contradiction, that $E$ contains a cycle $c$ with an exit $f$ at a
vertex $v$. Consider the set $\{\varepsilon_{ij}=c^{i}ff^{\ast}(c^{\ast}%
)^{j}:1\leq i,j\leq n+1\}$. Clearly the $\varepsilon_{ij}$ form a set of
matrix units as $(\varepsilon_{ii})^{2}=\varepsilon_{ii}$ and $\varepsilon
_{ij}\varepsilon_{kl}=\varepsilon_{il}$ or $0$ according as $j=k$ or not. Then
it is easy to see that the set $S=\big \{
{\displaystyle\sum\limits_{i=1}^{n+1}} \, 
{\displaystyle\sum\limits_{j=1}^{n+1}}\,
k_{ij}\varepsilon_{ij}:k_{ij}\in K \big \}$ is a subring of $L$ isomorphic to the
matrix ring $\Ma_{n+1}(K)$. Since $S$ has bounded index of nilpotence $n+1$,
this is a contradiction. Hence no cycle in $E$ has an exit, thus proving (b). 

This finishes equivalence of all the statements. 
\end{proof}

\begin{remark}
We have seen that the notions of direct-finiteness and graded direct-finiteness coincide for Leavitt path algebras. However, this is not the case for $\Sigma$-$V$ and graded $\Sigma$-$V$ Leavitt path algebras as we will see in the examples below. 
\end{remark}

\begin{example}
If the graph $E_{1}$ consists of a single vertex
$v$ an a single loop $c$ based at $v$, then $L_{K}(E_{1})\cong K[x,x^{-1}]$
under the map $v\longmapsto1$, $c\longmapsto x$ and $c^{\ast}\longmapsto
x^{-1}$. As a $K[x,x^{-1}]$-module, $K[x,x^{-1}]$ is graded simple since
$\{0\}$ is the only proper graded simple module, but it is not a simple
$K[x,x^{-1}]$-module, as $K[x,x^{-1}]$ contains infinitely many ideals. Thus
trivially, $L_{K}(E_{1})$ is a graded $\Sigma$-$V$ ring, but is not a $\Sigma$-$V$ ring (actually, being a commutative integral domain, it is not von
Neumann regular and thus not even a $V$-ring). 
\end{example}

\begin{example}
If $E_{2}$
is a finite graph consisting of a cycle $c$ without exits, every path in
$E_{2}$ eventually end at a vertex on $c$ and if the integer $n$ is the total
number paths that end at a vertex on $c$ and does not include the entire cycle
$c$, then $L_{K}(E_{2})\cong \Ma_{n}(K[x,x^{-1}])$ under the natural grading of
matrices (see \cite{HR}) and is a graded artinian semi-simple ring being a
graded direct sum of finitely many graded simple modules isomorphic to
$K[x,x^{-1}]$. Thus $L_{K}(E_{2})$ is a graded $\Sigma$-$V$ ring, but is not a
$\Sigma$-$V$ ring.
\end{example}

It is easy to see that a graded homomorphic image of a graded $\Sigma$-$V$ ring is again a graded $\Sigma$-$V$ ring, but a graded extension of a graded $\Sigma$-$V$ ring by a graded $\Sigma$-$V$ ring need not be a graded $\Sigma$-$V$ ring even in the case of Leavitt path algebras, as the next example shows.

\begin{example} \rm
Let $E$ be the graph 
\[
\begin{array}
[c]{ccccccccc}%
\bullet_{v_{11}} & \overset{e_{11}}{\longrightarrow} & \bullet_{v_{12}} &
\overset{e_{12}}{\longrightarrow} & \bullet & \dashrightarrow & \cdot & \cdot
& \cdot\\
\uparrow &  & \uparrow &  & \uparrow &  &  &  & \\
\bullet_{v_{21}} & \longrightarrow & \bullet_{v_{22}} & \longrightarrow &
\bullet & \dashrightarrow & \cdot & \cdot & \cdot
\end{array}
\]

\noindent If $S=\langle v_{11} \rangle$ is the ideal generated by $v_{11}$, then $S$ is the
graded socle and $S\neq L_{K}(E)$ as $v_{2n}\notin S$ for all
$n\geq1$. Being graded semi-simple, $S$ is a graded $\Sigma$-$V$ ring and
$L_{K}(E)/S$ is also graded semi-simple and hence is a graded $\Sigma$-$V$ ring. But $L_{K}(E)$ is not a graded $\Sigma$-$V$ ring. To see this,
first note that, since $E$ is acyclic, by Theorem 2.10 in \cite{HR}, the graded socle $S\cong
_{\gr}\Ma_{\Lambda}(K)$. So $S$ is a graded
direct sum of isomorphic graded simple modules. If $L_{K}(E)$ is a graded $\Sigma$-$V$ ring, $S$ is graded injective and hence a graded direct summand of
$L_{K}(E)$. Then $L_{K}(E)\cong_{\gr}S\oplus(L_{K}(E)/S)$
is a graded direct sum of graded simple modules and so $L_{K}(E
)=\Soc^{\gr}(L_{K}(E))=S$, a contradiction. Hence $L_{K}(E)$ is not a
graded $\Sigma$-$V$ ring. 
\end{example}

Next we show how the ideas of Section $4$ without the use of
gradings give rise to the description of the Leavitt path algebra $L_K(E)$ of an
arbitrary graph $E$ which is a (not necessarily graded) $\Sigma$-$V$ ring. In
this case, the graph $E$ is shown to be acyclic so $L$ becomes von Neumann
regular and is a subdirect product of semi-simple rings.

Recall that a ring $R$ is said to be a left/right $V$-ring if every simple
left/right $R$-module is injective. We begin with establishing a necessary
condition for a Leavitt path algebra to be a $V$-ring. 

\begin{lemma}
\label{rwr} Suppose $R$ is a ring with local units. If $R$ is a right $V$-ring, then $R$ is right weakly regular,
that is, $I^{2}=I$ for any right ideal $I$ of $R$.
\end{lemma}

\begin{proof}
It is easy to check that a theorem of Villamayor on $V$-rings (see \cite{L}) holds for a $V$-ring
$R$ with local units, namely, every right ideal is an intersection of maximal
right ideals of $R$. Let $I$ be a non-zero right ideal of $R$. If $I\neq
I^{2}$, let $a\in I\backslash I^{2}$. Since $I^{2}$ is an intersection of
maximal right ideals, there is a maximal right ideal $M$ containing $I^{2}$
such that $a\notin M$. Then $R=aR+M$. Let $u$ be a local unit satisfying
$au=ua=a$. Write $u=ax+m$ where $x\in R$ and $m\in M$. Then $a=ua=axa+ma\in
I^{2}+M=M$, a contradiction. Hence $I^{2}=I$ for every right ideal $I$ of $R$.
\end{proof}

We need the following result from \cite{ARS}.

\begin{lemma}[Theorem 3.1, \cite{ARS}] \label{rwr=>K}  
Let $L:=L_{K}(E)$ be a Leavitt path algebra of an arbitrary graph $E$. If $I^{2}=I$ for
every right ideal $I$ of $L$, then the graph $E$ satisfies
Condition (K), that is, every vertex $v$ on a simple closed
path $c$ is also part of another simple closed path $c^{\prime}$  different from $c$.
\end{lemma}

\begin{proposition}
\label{sigmaV=>vo Neumann regular} If $L:=L_{K}(E)$ is a $\Sigma$-$V$ ring,
then the graph $E$ contains no cycles and $L$ is von Neumann regular.
\end{proposition}

\begin{proof}
Now the same proof of  Proposition \ref{grSigmaV=>df} without the grading
shows that $L$ is directly-finite and so, by Theorem \ref{munclassroom}, no closed path in
the graph $E$ has an exit. On the other hand, since $L$ is also a right
$V$-ring, Lemma \ref{rwr} \ and Lemma \ref{rwr=>K} imply that the graph $E$
satisfies Condition (K) which in particular implies that every cycle in $E$
has an exit. In view of these contradicting statements we conclude that the
graph $E$ contains no cycles. By~\cite[Theorem 1]{AR}  $L$ is
von Neumann regular.
\end{proof}

Now using Proposition
\ref{sigmaV=>vo Neumann regular} and repeating the ungraded version of the
proof of Theorem \ref{main Th}, we obtain the following description of Leavitt
path algebras which are $\Sigma$-$V$ rings.

\begin{theorem}
\label{Sigma-V-LPA} Let $E$ be an arbitrary graph. Then for $L:=L_{K}(E)$ the following
properties are equivalent.
\begin{enumerate}[\upshape(a)]
\item $L$ is a $\Sigma$-$V$ ring;
\item $L$ is von Neumann regular and, for each primitive ideal $P$ of $L$, $L/P\cong
\Ma_{n}(K)$, where $n$ could be possibly infinite.
\end{enumerate}
\end{theorem}

\begin{corollary}
If a Leavitt path algebra $L_K(E)$ is a $\Sigma$-$V$ ring, then it is a graded $\Sigma$-$V$ ring. 
\end{corollary}

\begin{proof}
Let $L_K(E)$ be a $\Sigma$-$V$ ring. By Proposition \ref{sigmaV=>vo Neumann regular}, $E$ is acyclic and by \cite[Theorem 5.1]{HR}, every one-sided ideal of $L_K(E)$ is graded and so above theorem states that, in this case, for every graded primitive ideal $P$, $L_K(E)/P$ is a matrix ring over $K$ and so by Theorem \ref{siginjecmm}, $L_K(E)$ is a graded $\Sigma$-$V$ ring.
\end{proof}

\begin{remark}
In view of Theorems \ref{main Th} and \ref{Sigma-V-LPA}, it is clear
that the well-known Leavitt ring $L(1,n)$ is not a $\Sigma$-$V$ ring and also not a
graded $\Sigma$-$V$ ring.
\end{remark}

\noindent As noted in the Introduction, Kaplansky showed that a commutative ring $R$ is a $V$-ring if and only if $R$ is von Neumann regular and that this result no longer holds for non-commutative rings. So it is natural to ask what happens to this result if $R$ is a Leavitt path algebra? Lemmas \ref{rwr}, \ref{rwr=>K} and Theorem \ref{DirFinLPAs} enable us to easily prove the following.

\begin{proposition}
If a Leavitt path algebra L is a right V-ring, then L is von Neumann regular if and only if L is directly-finite.
\end{proposition}

We finish the paper by asking the following two questions.

\begin{question}
If a Leavitt path algebra $L$ is a right $V$-ring, is it von Neumann regular?
\end{question}

\begin{question}
 If a Leavitt path algebra $L$ is von Neumann regular, is it a right $V$-ring?
\end{question}

\bigskip

\noindent {\bf Acknowledgement.} We are thankful to the referee for not only a very careful reading of this paper, but also making many helpful suggestions and criticisms which led to substantial reorganization and improvement of this paper.

\end{document}